%% file: project_random_set.tex
\RequirePackage[l2tabu, orthodox]{nag}
\documentclass[preprint,  11pt]{amsart}

\usepackage{fullpage}
\usepackage{graphicx}
\usepackage{xcolor}
\usepackage{lipsum}
\usepackage{mwe}
\usepackage{tikz-network}
\usepackage[square,numbers]{natbib}

\usepackage{libertinus}
\usepackage[T1]{fontenc}
\usepackage{pdfpages}
\usepackage{pgfplots}
\usetikzlibrary{positioning}
\usepackage{gfsartemisia}
\usepackage[T1]{fontenc}
\usepackage{subfigure}
\usepackage{subcaption}
\input{mypreamble_2018}

\renewcommand\P{\mathbb P}
\renewcommand\E{\mathbb E}

\renewcommand{\bex}{\indent\begin{exercise}}
\renewcommand\subset{\subseteq}

\hypersetup{
    colorlinks=true, %set true if you want colored links
    linktoc=all,     %set to all if you want both sections and subsections linked
    linkcolor=blue,  %choose some color if you want links to stand out
}
\def\H{\mathbb H}

\begin{document}

\title{Partial divisibility of random sets and powers of completely monotone functions}
\author[J.~S.~Baslingker]{Jnaneshwar Baslingker}
\author[B.~Dan]{Biltu Dan}
\address{Department of Mathematics, Indian Institute of Science, Bangalore-560012, India.}
\email{jnaneshwarb@iisc.ac.in, biltudan@iisc.ac.in}
\date{\today}

%\begin{abstract}{\color{red}Needs to be rephrased}
%Let $V_{\XX}$ be the void functional of a random subset $\XX$ of $[n].$ We show that $V_{\XX}^\alp$ is again a void functional of some random subset of $[n]$ for all $\alp\geq n-1.$ Moreover, the lower bound is sharp. In particular, we show that if $\XX$ puts positive mass only on all singletons, then the set of $\alp$ for which $V_{\XX}^\alp$ is the void functional of some random subset of $[n]$ has the form $\{0, 1, \ldots, n-2\} \cup [n-1, \infty)$. With motivation from a question addressed in~\cite{BD2022}, we ask the question: what is the possible form of the set for a fixed random subset $\XX$ of $[n]$? We give examples of random sets for which the set of such $\alp$s consists of a finite set and more than one disjoint interval of positive length. Study of powers of void functionals of random sets is a special case of study of powers of completely monotone functions in general. In this article we study what powers preserve complete monotonicity of functions on finite lattices.
%\end{abstract}
\begin{abstract}
In this article, we study exponents which preserve complete monotonicity of functions on lattices. We prove that for any completely monotone function $f$ on a finite lattice, $f^\alpha$ is completely monotone for all $\alpha\ge c$, where $c$ is explicitly described. For finite distributive lattices we show that the bound $c$ is sharp. Important examples of completely monotone functions are void functionals of random closed sets. We prove that if $V_\XX$ is the void functional of a random subset $\XX$ of $[n]$, then $V_\XX^\alpha$ is void functional of some random closed set for $\alpha\ge n-1$. The results are analogous to the result of FitzGerald and Horn~\cite{FH77} on Hadamard powers of positive semi-definite matrices. Also, we study the question of approximating an $m$-divisible random set by infinitely divisible random sets, and its generalization to lattices.
\end{abstract}

\maketitle

%\tableofcontents

\section{Introduction}

%\section{Choquet's theory of capacities: Notes from Molchanov and Matheron}

Completely monotone functions, which are Laplace transforms of positive measures, have been extensively studied. Completely monotone functions (c.m.\ functions) are an important class of functions with significant applications in various branches of mathematics. They find uses in potential theory~\cite{bergfrost}, probability theory~\cite{bondesson1992,feller1966,Kimberling}, physics~\cite{day1970}, numerical and asymptotic analysis~\cite{Frenzen,wimp1981}, and combinatorics, among other areas. The monograph by Widder~\cite{Widder} provides a comprehensive collection of important properties of c.m.\ functions.

A function $f: (0, \infty) \rightarrow [0,\infty)$ is said to be completely monotone if it has derivatives of all orders and satisfies the condition:
\begin{align*}\label{eq: c.m.\ continous defn}
\left(-\frac{d}{dx}\right)^n f(x) \geq 0 \quad \text{for all } x > 0 \text{ and } n = 0, 1, 2, \ldots
\end{align*}

Similarly, a sequence $a=\{a_n\}_{n\geq 0}$ is said to be completely monotone sequence (c.m.\ sequence) if 
\begin{align*}
((-D)^ka)_j \geq 0, \mbox{for all } k,j\geq 0  
\end{align*}
where for any sequence $b=\{b_n\}_{n\geq 0}$ we define $Db$ to be the sequence $\{b_{n+1}-b_n\}_{n\geq 0}$. A classical result in analysis asserts that c.m.\ sequences are nothing but moment sequences of finite positive measures on $[0,1]$ (see the Hausdorff moment sequence theorem, Proposition $6.11$ of Chapter~$4$ of~\cite{BCRR2012harmonic}).
%involving c.m.\ sequences is the Hausdorff moment sequence theorem (see~\cite{Hausdorff}). The theorem states that a sequence of numbers is the moment sequence of a finite positive measure on $[0,1]$ if and only if the sequence is c.m. 
One notable result in the theory of c.m.\ functions is Bernstein's theorem (see Theorem $6.13$ of Chapter~$4$ of~\cite{BCRR2012harmonic}), which characterizes c.m.\ functions as Laplace transforms of positive measures on $[0,\infty)$. In other words, $f$ is a c.m.\ function if and only if $f(t)=\int\limits_{[0,\infty)}\exp(-tx)d\mu(x)$ for some finite positive Borel measure $\mu$ on $[0,\infty)$. An important application of complete monotonicity in probability theory is in the study of void functionals of random closed sets. 
% This theorem provides a detailed explanation of the relationship between completely monotone functions and moments, particularly in the context of measures on commutative semigroups. A comprehensive treatment of this theorem can be found in the book by Berg, Christensen, and Ressel [3].

 \subsection{Void functionals of random closed sets}
 
 % The concept of random set was mentioned for the first time together with the mathematical foundations of Probability Theory(add Kolmogorov's ref). Further progress in this theory relied on developments in the following areas: studies of random elements in general topological spaces, in groups and semigroups(add REF), in general theory of stochastic processes and the theory of capacities, set valued analysis and multifunctions, advances in image analysis and microscopy. The mathematical theory of random sets can be traced back to Matheron and Kendall. We refer the reader to the monograph by Molchanov for a detailed study on random sets.

 The origin of the modern concept of a random set goes as far back as the seminal book by A.N. Kolmogorov~\cite{kolmogorov}. Further progress in the theory was due to developments in the areas such as the study of random elements in general topological spaces, groups, and semigroups~\cite{grenander2008}, the general theory of stochastic processes~\cite{dellacherie}, point processes~\cite{kallenberg2017random}, potential theory~\cite{choquet}, advances in image analysis and microscopy~\cite{serra}. We refer the reader to the monographs by Matheron~\cite{matheron1974} and Molchanov~\cite{molchanov2017} for a comprehensive study on random sets.

 Let $E$ be a locally compact, Hausdorff, second countable topological space. Let $\FF,\KK$ be the collections of closed, compact subsets of $E$ respectively.   

Random sets in this article will always mean random closed sets. A \textit{random closed set} in $E$ is a random variable $\XX$ (on some probability space $(\Omega, \mathfrak{F},\mathbb P)$) taking values in $\FF$ and measurable w.r.t.\ $\BB(\FF)$.  Here $\BB(\FF)$ is the Borel sigma-algebra, defined with respect to the Fell topology on $\FF$ (see~\cite{matheron1974, molchanov2017}). We mention here that if $E$ is a compact metric space, then the Hausdorff metric metrizes the Fell topology on $\FF$.

The distribution of a random closed set $\XX$ is determined by its {\em void functional} $V_{\XX}:\KK\mapsto [0,1]$ defined as
\begin{equation*}
\label{eq:capacityfnl}
V_{\XX}(K):=\P\{\XX\cap K =\emptyset\}.
\end{equation*}
The {\em capacity functional} is defined as $T_{\XX}:=1-V_{\XX}$, i.e., $T_{\XX}(K):=\P\{\XX\cap K\neq\emptyset\}$.

In this article, we study the positive powers of void functionals which continue to be void functionals. More precisely, we ask the following natural question. Let $V_{\XX}$ be the void functional of $\XX$. For which $\alp>0$ is $V_{\XX}^{\alpha}$ a void functional? 
%It may also be relevant to ask this power relationship only for a smaller class of sets $\KK'\subseteq \KK$.
%
If $V_{\YY}(K)=V_{\XX}(K)^{\alpha}$ for all $K\in\KK$, we denote $\YY$ as $\XX_{\alpha}$. This question is motivated by the following result of Lawler, Schramm and Werner (see Section $2.3$ of~\cite{werner}).  Let $\XX$ be the image of the Brownian motion in $\H$ (upper half plane) started at $0$ and conditioned to exit $\H$ at $\infty$. A {\em hull} is a set $A$ such that $A=\overline{A\cap\H}$ and $0,\infty\not\in A$ and $\H \setminus A$ is simply connected. For a hull $A$, there is a unique conformal map $\Phi_A$ from $\H\setminus A$ onto $\H$ that fixes $0$ and $\infty$ and such that $\Phi_A(z)\sim z$ as $z\to \infty$. \citet{virag} showed that $V_{\XX}(A)=\Phi_A'(0)$ when $A$ is a  hull. Then $\XX_{\alpha}$ (if it exists) must satisfy $V_{\XX_{\alpha}}(A)=\Phi_A'(0)^{\alpha}$ for all hulls $A$. Lawler, Schramm and Werner showed that a random set $\XX_{\alpha}$ exists if and only if $\alpha\ge \frac58$. Note that they did not consider all compact sets in $\overline{\H}$ but only hulls.

First we list some properties of the void functional. But before that we need to introduce some notions. Let $\mathcal{U}$ be a collection of subsets of $E$  that is closed under finite unions. For any functional $\varphi:\mathcal{U} \to \mathbb{R}$ and $A\in \mathcal{U}$, 
we define 
$\Delta_A\varphi:\mathcal{U}\rightarrow \R $ by $\Delta_A\varphi(B)=\varphi(B)-\varphi(A\cup B)$, which is a form of discrete derivative. For $A,A_1,\ldots ,A_n\in \mathcal{U}$, one can verify inductively that
% %
% \begin{equation*}
% \label{def:discderivative1}
% \Delta_{A_n}\ldots \Delta_{A_1}\varphi(A) := \Delta_{A_{n-1}}\ldots \Delta_{A_1}\varphi(A) - \Delta_{A_{n-1}}\ldots \Delta_{A_1}\varphi(A \cup A_n).
% \end{equation*}
%

%
\ba
\Delta_{A_n}\ldots \Delta_{A_1}\varphi(A)=\sum_{J\subseteq [n]}(-1)^{|J|}\varphi\l(A \cup \bigcup_{i\in J}A_i\r).
\ea
A functional $\varphi:\mathcal{U} \to \R$ is said to be {\em  completely monotone} if it is non-negative and $\Delta_{A_n}\ldots \Delta_{A_1}\varphi(A) \geq 0$ for all $n \geq 1$ and all $A,A_1,\ldots,A_n \in \mathcal{U}$. It is said to be {\em completely alternating} if  $\Delta_{A_n}\ldots \Delta_{A_1}\varphi(A) \leq 0$ for all $n \geq 1$ and all $A,A_1,\ldots,A_n \in \mathcal{U}$.
%A  functional  $\varphi:\DD\to \R$ is said to be {\em completely monotone} if if $\Delta_{A_1}\ldots \Delta_{A_n}\varphi(A) \geq 0$ for all $n \geq 1$ and all $A_1,\ldots,A_n \in \DD$. 
Note that the definition is analogous to that of continuous case. Also note that if $\varphi$ is completely monotone, then $\varphi$ is monotone, that is, $\varphi(A)\geq \varphi(B)$ if $A\subset B$. 
The void functional of a random closed set $\XX$ satisfies the following properties (see .
\benu
\item $V_{\XX}(\emptyset) = 1.$
% \item (Monotonicity) $V_{\XX}(K_2) \le V_{\XX}(K_1)$ if $K_1 \subset K_2$.
\item (l.s.c.\ i.e., lower semi-continuity\footnote{This can be shown to be equivalent to the usual formulation of lower semi-continuity via $\liminf$.}) If $K_n\downarrow K$, then $V_{\XX}(K_n) \uparrow V_{\XX}(K)$.
\item $V_{\XX}$ is c.m.\ on $\KK$. 
\eenu
G. Choquet proved that any functional satisfying above properties is void functional of some random set (see Theorem $1.1.29$~\cite{molchanov2017}). Thus, a functional $V:\KK\mapsto [0,1]$ satisfying $V_{\XX}(\emptyset) = 1$ is the void functional of a random set if and only if it is l.s.c.\ and c.m.

% %
%% \bthm[Theorem $1.1.29$~\cite{molchanov2017}] \label{thm:choquet} A functional $V:\KK\mapsto [0,1]$ satisfying $V_{\XX}(\emptyset) = 1$ is the void functional of a random set if and only if it is l.s.c. and c.m.\
% \ethm

 Since lower semi-continuity of $V_{\XX}^{\alpha}$ follows from that of $V_{\XX}$, the question of existence of $\XX_{\alpha}$ is really a question of complete monotonicity of $V_{\XX}^{\alpha}$. For $\alpha\in \N$, $\XX_{\alpha}$ does exist, and it is just the union of $\alpha$ i.i.d.\ copies of $\XX$. For fractional $\alpha$, it is far from obvious that $\XX_{\alpha}$ exists. Below we give two examples. In the first example, any $\alpha>0$ works, where as in the second example, $\XX_{\alpha}$ does not exist if $\alpha < 1$. We remark that the question of existence of $\XX_{\alpha}$ for any $\alpha>0$ has not been studied but the particular case of existence of $\XX_{1/m}$ for all $m\in\N$ has been studied. If $\XX_{1/m}$ exists for some $m\in\N$ then $\XX$ is called $m$-divisible and if $\XX_{1/m}$ exists for all $m\in\N$, then $\XX$ is called infinitely divisible (for more on infinitely divisible random sets see Chapter $4$ of~\cite{molchanov2017}). In view of this the existence of $\XX_\alp$ is referred to as partial divisibility of $\XX$.

 \begin{example}
      Let $\XX$ be defined as the Poisson point process in $\mathbb{R}^d$ with intensity measure $\lambda(\cdot)$. Then $V_\XX(K)=e^{-\lambda(K)}$. If $\YY$ denotes the Poisson point process with intensity $\alp \lambda(\cdot)$, then $V_{\YY}(K)=V_\XX(K)^{\alpha}$.
 \end{example}
 \begin{example}
     Let $E=\{1,2\}$. Define $\P(\XX=\{1\})=p$ and $\P(\XX=\{2\})=1-p$, where $0<p<1$. Then $\Delta_{\{1\}}\Delta_{\{2\}}V_{\XX}^{\alpha}(\emptyset)=1-p^\alpha-(1-p)^\alpha$ and this can be seen to be negative for $\alpha<1$. So, $\XX_\alpha$ does not exist for $0<\alpha<1$. For this example, $\XX_\alpha$ exists for $\alpha\geq 1$.
 \end{example}
%  The following theorem shows that $\alpha_n=n-1$. As l.s.c is preserved by taking exponents, the problem of finding the exponents that preserve void functionals, we need to study the exponents that preserve complete monotonicity of functions.

Powers of c.m.\ functions also appear in the study of infinitely divisible random sets under the operation of union. 

\subsection{Infinite divisibility}
A random variable $X$ taking values in $\mathbb{R}$ is said to be $m$-divisible if there exist i.i.d.\ random variables \(X_{m,1}, X_{m,2}, \ldots, X_{m,m}\) such that their sum \(X_{m,1} + X_{m,2} + \ldots + X_{m,m}\) has the same distribution as \(X\). It is said to be infinitely divisible if it is $m$-divisible for every positive integer \(m\). Let $\DD_{m}$ denote the set of $m$-divisible distributions and $\DD_{\infty}$ denote the set of infinitely divisible distributions on $\R$. Given two cumulative distribution functions \(F(x)\) and \(G(x)\) on $\mathbb{R}$, the Kolmogorov distance, denoted by \(\rho\), is calculated as:
\[\rho(F,G) := \sup_{x\in\mathbb{R}} |F(x) - G(x)|.\]

An important and beautiful result in the theory of infinitely divisible distributions is that 
\begin{align}\label{eq: arak ineq}
    C_1m^{-2/3}\leq \quad\sup_{F\in\DD_m}\inf_{G\in\DD_{\infty}} \rho(F,G)\quad\leq C_2m^{-2/3},
\end{align}
due to the brilliant work of \citet{arak1981,arak1982}, following a series of works by Kolmogorov, Prohorov, Mesalkin, Le Cam and others (see the review article~\cite{BoseDasgupta}). 

Motivated by this, we study a similar question for infinitely divisible random sets. The problem of approximating $m$-divisible random sets by infinitely divisible random sets, in turn has to do with approximating $m$-divisible c.m.\ functions by infinitely divisible c.m.\ functions (defined in the next sub-section). 

We study both of the problems, the powers of c.m.\ functions and approximating m-divisible c.m.\ functions by infinitely divisible c.m.\ functions, in the general setting of lattices.

\subsection{General set-up of lattices} \label{subsec: gen set up of lattices} Let $(\LL, \leq)$ be a partially ordered set. $\LL$ is said to be a \textit{lattice} if, for any two elements $a, b \in \LL$, there exist unique elements $x = a \wedge b$ (the meet or infimum) and $y = a \vee b$ (the join or supremum) such that:

1. $x \leq a$ and $x \leq b$,

2. $a \leq y$ and $b \leq y$,

3. For any lower bound $\ell$ satisfying $\ell \leq a$ and $\ell \leq b$, we have $\ell \leq x$.

4. For any upper bound $u$ satisfying $a \leq u$ and $b \leq u$, we have $y \leq u$.

 A lattice $(\LL, \leq)$ is said to be \textit{distributive} if for all $a,b,c\in\LL$,
\[
a \lor (b \land c) = (a \lor b) \land (a \lor c).
\]

    %  A poset $\LL$ is called a {\em lattice}
    % if each non-empty finite subset of $\LL$ has  a supremum and an infimum.{\color{red} define $\vee$, $\wedge$}.
    
   For more on lattices we refer the reader to Chapter 3 of~\cite{stanley2011enumerative}. We can define c.m.\ function on any lattice $\LL$. For any function $f:\LL \to \mathbb{R}_{\geq 0}$ and $x,y\in \LL$, 
define 
$\Delta_xf(y):= f(y)-f(x\vee y)$, a form of discrete derivative. Successively, if $x,x_1,\ldots ,x_n\in \LL$, then
\begin{equation*}
\label{def:discderivative}
\Delta_{x_n}\ldots \Delta_{x_1}f(x) := \Delta_{x_{n}}\ldots \Delta_{x_2}f(x) - \Delta_{x_{n}}\ldots \Delta_{x_2}f(x \vee x_1).
\end{equation*}
One can verify inductively that
\ba
\Delta_{x_n}\ldots \Delta_{x_1}f(x)=\sum_{J\subseteq [n]}(-1)^{|J|}f\l(x \vee \l(\vee_{i\in J}x_i\r)\r).
\ea
A function $f:\LL \to \mathbb{R}_{\geq 0}$ is said to be {\em  completely monotone} if  $\Delta_{x_n}\ldots \Delta_{x_1}f(x) \geq 0$ for all $n \geq 1$ and all $x,x_1,\ldots,x_n \in \LL$. As closed sets form a lattice (union and intersection are the join and meet operations), the definition of void functionals in the case of random sets is a special case of the above definition. More generally, c.m.\ functions have also been studied in the setting of semigroups (see Chapter $4$ of~\cite{BCRR2012harmonic}).

 We ask the following \textbf{question}: if $f$ is a c.m.\ function on a lattice, then for which $\alpha>0$ is $f^\alpha$ c.m.? We answer this question for finite lattices (see Theorem~\ref{thm:cm_lattice}). As void functional of a random subset of $[n]$ is c.m.\ on the lattice of subsets of $[n]$, the above question is a generalisation of our earlier question on powers of void functional being void functional. Note that if $f,g$ are c.m.\ on $\LL$, one can check that the product $fg$ is also c.m.\ It follows that if $f$ is c.m.\ then $f^{\alp}$ is c.m.\ for $\alp\in \N$, like in the case of c.m.\ functions on $(0,\infty)$.
%  For c.m.\ function $f$ on $(0,\infty)$, it is easy to see that $f^{\alp}$ is c.m.\ for any $\alp\in\N$.
% % Even for c.m.\ functions and c.m.\ sequences, interesting results appear (see Theorem ~\ref{thm: TJ trick}).
In Theorem~\ref{thm: TJ trick}, we give an example of c.m.\ function $f$ on $[0,\infty)$ for which $f^\alpha$ is c.m.\ only when $\alpha\in\N$. But if we consider the setting of finite lattices, then there are no such c.m.\ functions (see Theorem~\ref{thm:cm_lattice}).

Now we state the second question precisely. A random set $\XX$ is said to be $m$-divisible for union if there exist i.i.d.\ random sets $\XX_{m,1}, \XX_{m,2}, \ldots, \XX_{m,m}$ such that
\[\XX \stackrel{d}{=} \XX_{m,1} \cup \XX_{m,2} \cup \ldots \cup \XX_{m,m}.\] It is said to be infinitely divisible, if it is $m$-divisible for every \(m \geq 1\)  (see Chapter $4$ of~\cite{molchanov2017}). Infinitely divisible random sets can be characterised in terms of void functionals, which are of the form $e^{-(1-\psi)}$, where $\psi$ is roughly another void functional  (see Theorem $3\mbox{-}1\mbox{-}1$ of~\cite{matheron1974}). Also one can show that $\XX$ is an infinitely divisible random set if and only if $V_{\XX}^{1/m}$ is void functional of some random set, for each $m\geq 1$.

We define a c.m.\ function $f$ on lattice $\mathbb{L}$ to be $m$-divisible, $m\in \mathbb{N}$, if $f^{1/m}$ is c.m.\ We define a c.m.\ function $f$ to be infinitely divisible, if $f^{1/m}$ is c.m.\ for every $m\ge 1$.
Let $\mathcal F_m$ be the set of all $m$-divisible c.m.\ functions on a lattice $\mathbb{L}$ taking values in $[0,1]$ and $\mathcal F_{\infty}$ be the set of all infinitely divisible c.m.\ functions. 
Let 
\begin{align}\label{def: tau_n}
    \tau_m:=\sup\limits_{f\in \mathcal F_m}\inf\limits_{g\in \mathcal F_{\infty}} d(f,g),
\end{align}
where $$d(f,g):=\sup\limits_{x\in \mathbb{L}}|f(x)-g(x)|.$$

We ask the following \textbf{question}: does $\tau_m\rightarrow 0$, as $m\rightarrow\infty$? If yes, what is the exact rate at which $\tau_m \rightarrow 0$? We answer this question for any lattice which is not a chain (see Theorem~\ref{thm:cm_approx_lattice}). Note that in the case of lattice being a chain, any non-increasing, non-negative function is c.m.\ function and as a result, is infinitely divisible. We now present the main results of the article.

\section{Main results}
We present our main results in the following three subsections. In the first subsection we present our results on random sets. We then generalize them to the setting of lattices in the second subsection. In the third subsection, we present a few results which serve as examples related to the results of the first two subsections.

\subsection{Results on random sets:}

We first answer the question of existence of $\XX_{\alpha}.$ Recall that $\XX_{\alpha}$ is the random set with void functional $V_{\XX}^{\alpha}.$  We stick to the finite setting $E=[n]:=\{1,\ldots, n\}$, $n\in\N$.

\begin{theorem}\label{thm:alp_finitecase} If $\XX$ is a random subset of $[n]$, then $\XX_{\alpha}$ exists for any $\alpha\ge n-1$. If $\XX$ is the uniform random singleton set, then $\XX_{\alpha}$ does not exist for non-integer $\alpha<n-1$.
\end{theorem}

Theorem ~\ref{thm:alp_finitecase} (also Theorem~\ref{thm:cm_lattice}) is similar in nature to Theorem 2.2 of~\citet{FH77} on Hadamard powers of positive semi-definite (p.s.d.) matrices. They show that $\alpha$-th Hadamard power of  any $n\times n$ p.s.d.\ matrix with non negative entries is p.s.d.\ if $\alpha$ is an integer or $\alpha\geq n-2$. They also show that the bound $n-2$ is sharp. This result can be seen in parallel with Theorem ~\ref{thm:alp_finitecase}, where the sharp lower bound is $n-1$ instead of $n-2$. 

Next we present our result on approximating $m$-divisible random sets by infinitely divisible random sets.
Let $E$ be any locally compact, Hausdorff, second countable topological space. Let $\DD_m$ be the set of all $m$-divisible random sets in $E$ and $\DD_{\infty}$ be the set of all infinitely divisible random sets. 
Let $$\psi_m:=\sup\limits_{\XX\in \DD_m}\inf\limits_{\YY\in \DD_{\infty}} d_V(\XX,\YY),$$ where $$d_V(\XX,\YY):=\sup\limits_{K\in \KK}|V_\XX(K)-V_\YY(K)|.$$
Note that $\psi_{m}=0$ if $|E|=1$. So we consider $|E|>1$.
\begin{theorem}\label{thm:dist_void}
% There exists $c_1,c_2>0$ such that
Let $E$ be any locally compact, Hausdorff, second countable topological space with $|E|>1$. Then there exists $0<c_1,c_2<\infty$ such that for any $m\geq 1$,
\begin{align}\label{eq:n_div_approx}
\frac{c_2}{m}\leq  \psi_m\leq \frac{c_1}{m}.
\end{align}
\end{theorem}

\begin{remark}
   The above result gives that the exact rate of decay is $m^{-1}$, unlike in the case of \eqref{eq: arak ineq} where the rate is $m^{-2/3}$. Note that we obtained the optimal order $O(1/m)$ for the upper bound by considering the accompanying infinitely divisible random set, that is, the union of Poisson many i.i.d.\ copies of $\XX$. On the contrary, in the case of infinitely divisible distributions, the accompanying law does not give the optimal order for the upper bound obtained in~\cite{arak1981}. It follows from the proof that \begin{align*}
    \frac{1}{4\sqrt{e}(2+\sqrt{e})}\leq\liminf_{m}m\psi_{m}\leq\limsup_{m}m\psi_{m}\leq \frac{2}{e^2}.
\end{align*}
\end{remark}

\subsection{Results in the setting on lattices:}

In this subsection we present results that are generalization of the results in the previous subsection. We answer the first question on powers of c.m.\ functions (see Section ~\ref{subsec: gen set up of lattices}) for finite lattices. For any finite lattice $\LL$ and for any $x\in\LL$, let $d_x$ denote the number of covering elements of $x$ which is defined as 
\[ d_x:=\left|\{y\in\LL: y>x, \nexists z\in\LL \text{ with } x < z < y\}\right|.\]
Define $d_{max}:=\max\{d_x: x\in\LL\}$. We now state the following result which is a generalization (due to Choquet's Theorem, see Theorem $1.1.29$~\cite{molchanov2017}) of Theorem ~\ref{thm:alp_finitecase}. 
\begin{theorem}\label{thm:cm_lattice}

\hspace{2mm}
\begin{enumerate}
\item Let $\LL$ be any finite lattice.  Then for any c.m.\ function $f$ on $\LL$, the function $f^\alp$ is c.m.\ if $\alp$ is an integer or $\alp\ge d_{max}-1$.

\item Let $\LL$ be any finite distributive lattice. Then there exists a c.m.\ function $f$ on $\LL$ such that $f^\alp$ is not c.m.\ for any non-integer $\alp<d_{{max}}-1$. In other words, the set of $\alpha$ for which $f^\alp$ is c.m.\ for any c.m.\ function $f$ on $\LL$ is $\N \cup [d_{{max}}-1, \infty)$.
\end{enumerate}
\end{theorem}

%\begin{remark}
%Following the proof of Theorem ~\ref{thm:cm_lattice}, it is easy to see that product of two c.m.\ functions on any finite lattice is c.m.\ This also implies that product preserves complete monotonicity on any lattice.
%\end{remark}
We remark that for finite non-distributive lattices, part~$(2)$ of the above theorem does not hold. We give example of a non-distributive lattice for which there exists non-integer $\alpha < d_{max}-1$ such that $f^\alp$ is c.m.\ whenever $f$ is c.m.\ (see Example ~\ref{example:nondist}). For infinite lattice the answer to the question of complete monotonicity of powers of c.m.\ functions depends on the structure of the lattice. If the lattice is a chain, then it is trivial to see that any non-negative, non-decreasing function is c.m.\ and hence so is any non-negative power. Also, one can construct an infinite non-distributive lattice such that if $f$ is c.m.\ function then $f^{\alpha}$ is c.m.\ function for all $\alp\ge1$ and the lower bound is sharp (see Example ~\ref{example:nondist}).

%Let $\alp_{\LL}^*$ be the smallest $\alp$ such that $f^\alp$ is c.m.\ for $\alpha\ge \alp_{\LL}^*$ whenever $f$ is c.m.\ on the lattice $\LL$. Then Theorem~\ref{thm:cm_lattice} says that for finite distributive lattice we have $\alp_{\LL}^*=d_{max}-1$ where $d_{max}$ is defined as above. For finite non-distributive lattices we only have $\alp_{\LL}^*\le d_{max}-1$ (follows from Theorem~\ref{thm:cm_lattice}). But this upper bound is not tight (see Example ~\ref{example:nondist}).

%If the lattice is discrete and infinite, then $\alp_{\LL}^*$ depends on the structure of the lattice. One can construct an infinite non-distributive lattice such that if $f$ is c.m.\ function then $f^{\alpha}$ is c.m.\ function for all $\alp\ge1$ and the lower bound is sharp (see Example ~\ref{example:nondist}). But if the lattice is a chain, then it is trivial to see that any non-negative, non-decreasing function is c.m.\ And hence so is any non-negative power. 

%\begin{theorem}\label{thm:cm_distlattice}
%Let $\LL$ be any finite distributive lattice. Then for any c.m.\ function $f$ on $\LL$, the function $f^\alp$ is c.m.\ if $\alp$ is an integer or $\alp\ge d_{max}-1$. Furthermore, there exists a c.m.\ function $f$ on $\LL$ such that $f^\alp$ is not c.m.\ for any non-integer $\alp<d_{max}-1$. In other words, the set of $\alpha$ for which $f^\alp$ is c.m.\ for any c.m.\ function $f$ on $\LL$ is $\N \cup [d_{max}-1, \infty)$.
%\end{theorem}

To prove Theorem~\ref{thm:cm_lattice} we use the following characterization of c.m.\ functions on a lattice. This proposition is similar to Theorem $1.2.15$ of ~\cite{molchanov2017} about the bijection between c.m.\ functions on continuous lattice and locally finite measures on lattice.
\begin{proposition}\label{prop:p-g}
Let $\LL$ be a finite lattice. A function $g:\LL\to [0,\infty)$ is c.m.\ if and only if there exists a function $p:\LL\to [0,\infty)$ such that 
\begin{align}\label{eq:p-g} g(x)=\sum_{y\ge x} p(y).\end{align}
\end{proposition}

%\begin{remark}
%For any $\alp>0$, we define a random variable $X$ to be $\alp$-divisible under addition if $\forall t>0$, we have $f^{1/\alp}(t)=\E[\exp(-tY)]$ for some random variable $Y$, where $f(t):=\E[\exp(-tX)]$. Note that the above partial divisibility (for non-integer values of $\alpha$) generalizes the classical $n$-divisibility of random variables for $n\in\N$. Similarly we can define $\alpha$-divisibility under products. For any $\alp>0$, we define a random variable $X \ge 0$ to be $\alp$-divisible under products if $\forall k\in\N$, we have $f^{1/\alp}(k)=\E[Y^k]$ for some random variable $Y$, where $f(k):=\E[X^k]$. It can be noted that Bernstein's theorem, Hausdorff moment theorem and Choquet theorem are used to relate powers of c.m.\ functions to divisibility for random variables under addition, multiplication and for random sets under union respectively.
%\end{remark}

% In the case of powers of void functionals, we have the following result, which is a special case of Theorem~\ref{thm:cm_lattice}, where we consider the distributive lattice to be Boolean lattice of subsets of $[n]$. In fact, we first prove this result and use it to prove the more general result in Theorem~\ref{thm:cm_lattice}.

We answer the second question (see Section ~\ref{subsec: gen set up of lattices}) on approximating $m$-divisible c.m.\ functions by infinitely divisible c.m.\ functions, for any lattice (not necessarily finite) that is not a chain. The following result is a generalization of Theorem ~\ref{thm:dist_void}.

\begin{theorem}\label{thm:cm_approx_lattice}
Let $\LL$ be any lattice which is not a chain. Let $\tau_m$ be as in \eqref{def: tau_n}. Then for any $m\geq 1$,
\begin{align*}\label{eq:n_div_approx1}
\frac{c_2}{m}\leq  \tau_m\leq \frac{c_1}{m}
\end{align*}
where the constants $c_1,c_2$ are as in Theorem~\ref{thm:dist_void}.
\end{theorem}

\subsection{Examples related to the results in previous subsections:}

In this subsection we present a few results which serve as examples related to the results in the previous subsections.\\ 

\paragraph{\textit {A counter-example for Theorem~\ref{thm:cm_lattice}}} First we give an example to show that part~$(2)$ of Theorem~\ref{thm:cm_lattice} does not hold for finite non-distributive lattices.  

\begin{example}\label{example:nondist}
Consider the non-distributive lattice $\LL$ shown in Figure~\ref{fig:1}.

%\begin{figure}[h!]
%    \centering
%    \includegraphics[scale=0.2]{non dist.png}
%    \caption{Non distributive lattice $L$}
%    \label{fig:my_label}
%\end{figure}

\begin{figure}[h!]
\begin{tikzpicture}
 \Vertex[label=$0$, size=0.4, x=3,y=0]{0} 
  \Vertex[label=$a$, size=0.4, x=0,y=3]{a} 
    \Vertex[label=$b$, size=0.4, x=3,y=3]{b} 
     \Vertex[label=$c$, size=0.4, x=6,y=3]{c} 
     \Vertex[label=$i$, size=0.4, x=3,y=6]{i} 
  \Edge[bend=0,Direct](0)(a)
  \Edge[bend=0,Direct](0)(b)
  \Edge[bend=0,Direct](0)(c)
  \Edge[bend=0,Direct](c)(i)
  \Edge[bend=0,Direct](b)(i)
  \Edge[bend=0,Direct](a)(i)
\end{tikzpicture}
\caption{Non distributive lattice $\LL$}
\label{fig:1}
\end{figure}

Here $d_{max}=3$. We will show that for any c.m.\ function $f$ on $\LL$, the function $f^{\alpha}$ is c.m.\ for any $\alpha\geq 1$.
% Using Theorem ~\ref{thm:cm_lattice}, we get the upper bound  $\alpha_{\LL}^*\leq 2$
Let $f$ be a c.m.\ function on $\LL$ and $p$ be the corresponding function given by Proposition ~\ref{prop:p-g}. We want to show that $f^{\alpha}$ is c.m.\ $\forall \alpha\geq 1$. For that it is enough to show  $h(\alpha)=(p_0+p_a+p_b+p_c+p_i)^{\alpha}-(p_a+p_i)^{\alpha}-(p_b+p_i)^{\alpha}-(p_c+p_i)^{\alpha}+2(p_i)^{\alpha} \ge 0$ 
 for any $\alp\geq 1$ (using Proposition ~\ref{prop:p-g}). Using Laguerre's sign change argument (Proposition $3.2$ of~\cite{jain}), it is easy to see that $h(\alpha)\geq 0$  for $\alpha\geq 1$. One can check that for $p_0=0$, $h(\alpha)<0$ for $0<\alpha<1$.
 
 Note that a similar argument works even if there are infinitely many covering elements of lattice point $0$ instead of $\{a,b,c\}$.
\end{example}

\paragraph{\textit{Examples related to Theorem~\ref{thm:alp_finitecase}}} Next we study a question analogous to a question addressed in~\cite{BD2022}. For a random set $\XX\subseteq [n]$, let $S_{\XX}$ denote the set of $\alpha>0$ for which $\XX_{\alpha}$ exists. Then by Theorem~\ref{thm:alp_finitecase} we have $\cap_{\XX\subset [n]}S_\XX=\{1,\ldots, n-2\}\cup [n-1, \infty)$.
%We also write $S_V$, if $V=V_{\XX}$. 
 We ask the following question for a fixed random subset $\XX$ of $[n]$: is the set of $\alp$ for which $\XX_{\alp}$ exists necessarily of the form $F \cup [\alp_*,\infty)$, where $F$ is a finite set? For many standard examples of $\XX$ the set $S_\XX$ turns out to be union of a finite set and a semi-infinite interval. We show by giving example that $S_\XX$ need not always be of that form. One can check (using~\eqref{eq: existence of random set}) that if $\XX$ is a random subset of $[n]$ and $1\leq n \leq 3$, then $S_\XX$ is union of a finite set and a semi-infinite interval. For $n \ge 4$, we give examples of random set $\XX\subset [n]$ such that $S_\XX$ has at least two interval components of positive length. The following theorem is analogous to Theorem $1$ in~\cite{BD2022}.

 \begin{theorem} \label{thm:multipleintervals}
Fix $n\ge 4$. For any $2\leq k\leq n-2$, there exists $\delta>0$ and a random subset $\XX^{(k)}$ of $[n]$ with $Q_{k}(A):=\P(\XX^{(k)}=A)$ depending only on $|A|$ for any $A\subset [n]$ such that the following holds.
\begin{enumerate}
\item $Q_{k}(A)=0$ if and only if $A=\emptyset$ or $1<|A|\leq n-k+1$
    \item $\XX_{\alpha}^{(k)}$ does not exist if $\alpha\leq n-k-1$ and $\alpha$ is non-integer
    \item $\XX_{\alpha}^{(k)}$ exists when $\alpha\in [j,j+\delta)$, $\forall n-k\leq j\leq n-2$
    \item $\XX_{\alpha}^{(k)}$ does not exist for some $\alpha_j\in (j,j+1), \forall n-k\leq j\leq n-2$
    \item $r_{k,B}(\alpha):=\sum_{A\subset B} (-1)^{|B|-|A|} \P\{\XX^{(k)} \subset A\}^\alp >0,$ for $\alpha\in(j-\delta, j+\delta), \forall 1\leq j\leq n-2,$ if $|B|\geq n-k+2$.
\end{enumerate}
\end{theorem}
The above theorem says that for fixed $n\ge 4$ and $2\le k\le n-2$, the set $S_{\mathcal X^{(k)}}=\{\alp\ge 0: \mathcal X^{(k)}_\alp \mbox{ exists}\}$ has at least $k$ interval components of positive length; for each of the integers from $n-k$ to $n-2$ there is one interval containing it and the last one is the semi-infinite interval containing
$n-1$. Furthermore, $S_{\mathcal X^{(k)}}\cap [0, n-k-1]=\{0, 1, \ldots, n-k-1\}$.
 
We mentioned that for many standard examples of $\XX$ the set $S_\XX$ turns out to be union of a finite set and a semi-infinite interval. It is natural to look for a class of random sets $\XX$ for which $S_\XX$ has only one interval component of positive length. In this direction we have the following result. It says that if $\XX \subset [n]$ has positive mass (not necessarily uniform) only on singletons, then $S_\XX$ has only one interval component of positive length.
\begin{theorem}\label{thm:oneinterval}
Fix $n\in\N$. Let $0<p_1,\ldots,p_n< 1$ such that $\sum_{i=1}^n p_i=1$. Let $\XX$ be a random subset of $[n]$ with $\P\{\XX=\{i\}\}=p_i$ for $i=1,\ldots,n$. Then $\XX_\alp$ exists if and only if $\alp\in\{0,1,\ldots,n-2\}\cup [n-1, \infty)$.
\end{theorem}
 The above theorem is similar to Theorem $1.1$ of~\citet{jain} on Hadamard powers of p.s.d.\ matrices.\\

%\section{$\alpha^*$ for non-distributive lattices}
% If $0<\alpha<1$, then considering
%  the distributive sub-lattice $\{0,a,b,i\}$ and applying Theorem ~\ref{thm:cm_lattice}, one can find a c.m.\ function $f$ on $\LL$ we find that $1\leq\alpha_{\LL}^*$. Thus we have $\alpha_{\LL}^*=1$.

\paragraph{\textit{Special examples of c.m.\ sequence and c.m.\ function}} We mention that if $\{a_n\}_{n\geq 0}$ is a c.m.\ sequence, then $\{a_n^\beta\}_{n\geq 0}$ is also c.m.\ sequence for $\beta\in\N$. Indeed, by Hausdorff moment sequence theorem, $\{a_n\}_{n\geq 0}$ is moment sequence (up to scaling) of a random variable $Z\in[0,1]$. Then $\{a_n^\beta\}_{n\geq 0}$ is moment sequence (up to scaling) of $\prod_{i=1}^\beta Z_i$, where $Z_1,\dots,Z_\beta$ are i.i.d.\ copies of $Z$. We have the following result for c.m.\ sequences and c.m.\ functions, which may be of independent interest.

\begin{theorem}\label{thm: TJ trick}
There exists $f:\N\rightarrow [0,\infty)$ such that $f$ is c.m.\ sequence and $f^{\alpha}$ is c.m.\ sequence if and only if $\alp\in\N$.

Also, there exists $g:(0,\infty)\rightarrow [0,\infty)$ such that $g$ is a c.m.\ function and $g^\alp$ is c.m.\ if and only if $\alpha\in\N$.
\end{theorem}

\paragraph{\textit{Outline of the rest of the paper}} We first prove Theorem~\ref{thm:alp_finitecase} in Section~\ref{sec: Proof of alp_finitecase}.
 We then prove
Proposition~\ref{prop:p-g} and use Theorem~\ref{thm:alp_finitecase} to complete the proof of Theorem~\ref{thm:cm_lattice} in Section~\ref{sec: proof of prop p-g}. In Section~\ref{sec: proof of dist} we first prove Theorem~\ref{thm:dist_void} using the fact that Poisson mixture of i.i.d.\ random sets is infinitely divisible to get the optimal upper bound. Then we use a sub-lattice of four elements to get the lower bound. We then extend the optimal approximation rate to Theorem~\ref{thm:cm_approx_lattice}. The proofs of Theorems~\ref{thm:multipleintervals} and~\ref{thm:oneinterval} are given in Section~\ref{sec: proof of one, multiple}.  We prove Theorem~\ref{thm: TJ trick} in Section~\ref{sec: proof of TJ trick}. The proof is independent of other proofs.

\section{Proof of Theorem~\ref{thm:alp_finitecase}}\label{sec: Proof of alp_finitecase}
%
% We first give the idea behind the proof of Theorem~\ref{thm:alp_finitecase}. We can sample a set with given void probabilities $V:2^{[n]}\mapsto [0,1]$ as follows. We condition on the status of $n$, i.e., on the event $n \in \XX$ and its complement. Let $\XX'=\XX\cap [n-1]$. Then
% \benu
% \item $\P\{n\notin\XX\}=V(\{n\})$ and conditional on this event, $\XX'$ has void probabilities \[ V_0(B)=P\{\XX\cap B=\emptyset\, \lvert \,n\notin \XX\}= \frac{V(B\cup\{n\})}{V(\{n\})},\quad B\subset [n-1].\]
% \item $\P\{n\in\XX\}=1-V(\{n\})$ and conditional on this event, $\XX'$ has void probabilities \[ V_1(B)=P\{\XX\cap B=\emptyset\, | \,n\in \XX\}=\frac{V(B)-V(B\cup\{n\})}{1-V(\{n\})},\quad B\subset [n-1].\]
% \eenu
% Continuing inductively, we can produce a sample of $\XX$. Another way to state it is to sample random sets $\XX_0$ and $\XX_1$ of $[n-1]$ with void probabilities $V_0$ and $V_1$, and an independent Bernoulli variable $\xi\sim \mb{Ber}(1-V(\{n\}))$, and set $\XX=\XX_1\cup\{n\}$ if $\xi=1$ and $\XX=\XX_0$ if $\xi=0$.

% Therefore, the process $\XX_{\alpha}$ exists if the processes $\XX_{0,\alpha}$ and $\XX_{1,\alpha}$ on $[n-1]$ having void probabilities  
% \[
% \frac{V^{\alpha}(B\cup\{n\})}{V^{\alpha}(\{n\})} \qquad \mb{ and } \qquad \frac{V^{\alpha}(B)-V^{\alpha}(B\cup\{n\})}{1-V^{\alpha}(\{n\})}
% \]
% exist. The existence of these processes on $[n-1]$ can be proved by induction. We now give the proof of the theorem.

We prove the first part of the theorem. The second part follows from the stronger result Theorem~\ref{thm:oneinterval}. We prove the first part by induction on $n$.
By Choquet's theorem (see Theorem $1.1.29$~\cite{molchanov2017}), it is enough to prove that if $V$ is the void functional of a random subset $\XX$ of $[n]$, then the functional $V^\alp$ is l.s.c.\ and c.m.\ The lower semi-continuity of $V^\alp$ follows from that of $V$. So, it suffices to prove that if $V$ is a c.m.\ function on $2^{[n]}$, then $V^\alp$ is c.m.\ for all $\alp\ge n-1$. For $n=1$ it is easy to check from the definition. We assume that the fact is true for $n=m-1$ with $m\ge 2$. Let $V$ be a c.m.\ function on $2^{[m]}$ and let $\alp \ge m-1$. We want to prove that $V^\alp$ is c.m.\ For $B\subset [m]$ we write
\begin{align*}
  V^\alp(B)= V^{\alpha}(\{m\}) \frac{V^{\alpha}(B\cup\{m\})}{V^{\alpha}(\{m\})} + (1-V^{\alpha}(\{m\})) \frac{V^{\alpha}(B)-V^{\alpha}(B\cup\{m\})}{1-V^{\alpha}(\{m\})}
\end{align*}
Observe that it is enough to prove that the functional $U$ and $W$ defined on $2^{[m]}$ by
\[
U(B):=\frac{V^{\alpha}(B\cup\{m\})}{V^{\alpha}(\{m\})} \qquad \mb{ and } \qquad W(B):= \frac{V^{\alpha}(B)-V^{\alpha}(B\cup\{m\})}{1-V^{\alpha}(\{m\})}
\]
are c.m.\

First, we show that $U$ is c.m.\ It is easy to see that the function $B\mapsto \frac{V(B\cup\{m\})}{V(\{m\})}$ is c.m.\ on $2^{[m-1]}$. Since $\alp \ge m-1>m-2$, $U$ is c.m.\ on $2^{[m-1]}$ by the induction hypothesis. Now for any $B_0, B_1,\ldots, B_k \in 2^{[m]}$ we have
\[
\Delta_{B_k}\ldots\Delta_{B_1}U(B_0)= \Delta_{B_k\setminus \{m\}}\ldots\Delta_{B_1\setminus \{m\}}U(B_0\setminus \{m\}) \ge 0.
\]
Thus $U$ is c.m.\ on $2^{[m]}$.

To show that $W$ is c.m.\ on $2^{[m]}$ we first show that $W$ is c.m.\ on $2^{[m-1]}$. Here we use a trick which was used by~\citet{FH77}. For $\beta\geq 1$ and $a,b\geq 0$, an explicit evaluation of the integral shows that,
\begin{align*}
    a^{\beta}-b^{\beta}=\beta \int_{0}^{1}(a-b)(ta+(1-t)b)^{\beta-1}dt.
\end{align*}
Using this, for $B\subset [m-1]$ we write
\ba
W(B)&=\frac{\alpha}{1-V^{\alpha}(\{m\})}\int_0^1(V(B)-V(B\cup \{m\})[\lambda V(B)+(1-\lambda)V(B\cup \{m\})]^{\alpha-1}d\lambda. \\
\ea
It is easy to see from the definition that the function $B\mapsto (V(B)-V(B\cup \{m\})=\Delta_{\{m\}}V(B)$ is c.m.\ on $2^{[m-1]}$. Also, for any $\lambda\in [0, 1]$, the function $B \mapsto \lambda V(B)+(1-\lambda)V(B\cup \{m\})$ is c.m.\ on $2^{[m-1]}$. Now since $\alp-1\ge m-2$, the second factor in the integrand is c.m.\ on $2^{[m-1]}$ by the induction hypothesis. Using the fact that the product of two void functionals is a void functional, one can show that $W$ is c.m.\ on $2^{[m-1]}$. Now suppose $B_0, B_1, \ldots, B_k\in 2^{[m]}$. If $m\in B_0$, then $\Delta_{B_k}\ldots\Delta_{B_1}W(B_0)=0$. Suppose $m \notin B_0$. Without loss of generality, let $m\notin B_i$ for $i=0,\ldots \ell$ and $m\in B_i$ for $i=\ell+1,\ldots,k$, where $\ell\geq 0$. Then we have
\[
\Delta_{B_k}\ldots\Delta_{B_1}W(B_0)=\Delta_{B_\ell}\ldots\Delta_{B_1}W(B_0)\ge 0
\]
Thus $W$ is c.m.\ on $2^{[m]}$ and this completes the proof.

\section{Proofs of Proposition ~\ref{prop:p-g} and Theorem~\ref{thm:cm_lattice}}\label{sec: proof of prop p-g}
We first prove Proposition~\ref{prop:p-g}.
\begin{proof}[Proof of Proposition~\ref{prop:p-g}]
Let $g$ be a c.m.\ function on $\LL$. We define a function $p$ on $\LL$ inductively as follows. For the maximum element $m$ of $\LL$ define $p(m):=g(m)$. Inductively, having defined $p(y)$ for all $y>x$, define \[p(x):=g(x)-\sum_{y>x} p(y).\]
By the definition, $p$ satisfies~\eqref{eq:p-g}. We now use the complete monotonicity of $g$ to show that $p$ is non-negative. Fix $x\in \LL$. Let $x_1,\ldots,x_r$ be all the covering elements of $x$, that is, 
\[
\{y\in\LL: y>x, \nexists z\in\LL \text{ with } x < z < y\}= \{x_1,\ldots,x_r\}.\]
Then we have
\ba
p(x)&=g(x)-\sum_{y>x} p(y)\\
&=g(x)-\sum_{i=1}^r \sum_{y\geq x_i} p(y) + \sum_{1\le i<j\le r}\sum_{y\geq x_i\vee x_j} p(y) -\sum_{1\le i<j<k\le r}\sum_{y\geq x_i\vee x_j\vee x_k} p(y) + \ldots\\
&=g(x)-\sum_{i=1}^r g(x_i) + \sum_{1\le i<j\le r}g(x_i\vee x_j) -\sum_{1\le i<j<k\le r}g(x_i\vee x_j\vee x_k) + \ldots\\
&=\Delta_{x_r}\ldots\Delta_{x_1}g(x) \ge 0.
\ea
To prove the converse, suppose $g$ is any non-negative function on $\LL$ and suppose there exists a non-negative function $p$ on $\LL$ such that \eqref{eq:p-g} holds. We want to prove that $g$ is c.m.\ Let $y, y_1, \ldots, y_k\in\LL$. For any $J\subset[k]$, we denote $y_J=\bigvee_{j\in J}y_j$. We have
\ba
\Delta_{y_k}\ldots\Delta_{y_1}g(y)&= \sum_{J\subset[k]}(-1)^{|J|}g\left(y\vee y_J\right)\\
&=\sum_{J\subset[k]}(-1)^{|J|} \sum_{z\ge y\vee y_J} p(z)\\
&= \sum_{z\ge y}p(z)\sum_{\{J\subset[k]:z\ge y\vee y_J\}}(-1)^{|J|}.
\ea
If $z\ge y$ but $z\ngeq y\vee y_i$ for any $i\in [k]$, then $\sum_{J\subset[k],z\ge y\vee y_J}(-1)^{|J|}=1$. Now suppose $z\geq y\vee y_i$ for some $i\in [k]$. Let $S_z$ be the largest subset of $[k]$ such that $z\ge y\vee y_{S_z}$. Then $z\ge y\vee y_J$ for any subset $J$ of $S_z$. Hence 
\[\sum_{J\subset[k],z\ge y\vee y_J}(-1)^{|J|} = \sum_{J\subset S_z}(-1)^{|J|}= \sum_{i=0}^{|S_z|} (-1)^i\binom{|S_z|}{i}=0.\]
Thus we have proved that $\Delta_{y_k}\ldots\Delta_{y_1}g(y)\ge 0$.
\end{proof}

To prove Theorem ~\ref{thm:cm_lattice} we use the following property of distributive lattices.
\begin{proposition}\label{prop:distlatticeproperty}
Let $\LL$ be a finite distributive lattice. Fix $x\in\LL$. Let $x_1,\ldots,x_n$ be all the covering elements of $x$.
% that is, 
% %
% \[
% \{y\in\LL: y>x, \nexists z\in\LL \text{ with } x < z < y\}= \{x_1,\ldots,x_n\}.\]
%
Then for any $1\le i_1<\ldots<i_\ell\le n$ and $1\le j_1<\ldots<j_m\le n$ with $1\le \ell \le m\le n$ and $\{i_1,\dots,i_{\ell}\}\neq \{j_1,\dots,j_{m}\}$ we have
\[x_{i_1}\vee\ldots\vee x_{i_\ell} \neq x_{j_1}\vee\ldots\vee x_{j_m}.\]
\end{proposition}
\begin{proof}
We use the following characterisation of distributive lattices: A lattice $\LL$ is distributive if and only if $x\vee y=x\vee z$ and $x\wedge y=x\wedge z$ imply $y=z$ for $x,y,z\in\LL$ (Corollary $1$, Chapter IX of~\cite{GB}). 

We prove the statement of the proposition by induction on $\ell$. Suppose $\ell=1$. The case when $m=1$ is trivial. Without loss of generality, suppose $x_i=x_1\vee x_2$ for some $i\in\{1,\ldots, n\}$, if possible. Suppose $i=1$. Then we have $x < x_2 < x_1$, which is not possible. Similarly, $i=2$ is not possible. So, let $i\in\{3,\ldots,n\}$. In this case, we have $x < x_1 <x_i$, which is not possible. Similar argument works for $m>2$. Thus the statement is true for $\ell=1$. 

Now suppose the statement is true for $\ell=1,\ldots, k$ with $1\le k<n$. We want to prove that the statement is true for $\ell=k+1$. Without loss of generality, if possible, suppose 
\[ x_1\vee\ldots \vee x_{k+1}= x_{i_1}\vee\ldots\vee x_{i_m}\]
for some $m\ge k+1$ and $1\le i_1<\ldots<i_m\le n$. We consider the following two cases. 

First suppose $\{1,\ldots, k+1\}\cap \{i_1, \ldots, i_m\} \neq \emptyset$. Without loss of generality, let $1=i_1$. Then we have
\[x_1\vee(x_2\vee\ldots\vee x_{k+1})=x_1\vee(x_{i_2}\vee\ldots\vee x_{i_m}).\]
Also 
\[x_1\wedge(x_2\vee\ldots\vee x_{k+1})=(x_1\wedge x_2)\vee\ldots\vee (x_1\wedge x_{k+1})=x\]
and
\[x_1\wedge(x_{i_2}\vee\ldots\vee x_{i_m})=(x_1\wedge x_{i_2})\vee\ldots\vee (x_1\wedge x_{i_m})=x.\]
Therefore by the above mentioned characterisation of distributive lattices we have
\[x_2\vee\ldots\vee x_{k+1}=x_{i_2}\vee\ldots\vee x_{i_m}.\]
But this is not possible by the induction hypothesis. 

Next suppose $\{1,\ldots, k+1\}\cap \{i_1, \ldots, i_m\}=\emptyset$. Then we have
\[x_1\vee(x_2\vee\ldots\vee x_{k+1})= x_1\vee(x_2\vee\ldots\vee x_{k+1}\vee x_{i_1}\vee\ldots\vee x_{i_m}).\]
Again, since $1\notin\{2,\ldots,k+1, i_1,\ldots, i_m\}$, we have
\[x_1\wedge(x_2\vee\ldots\vee x_{k+1}\vee x_{i_1}\vee\ldots\vee x_{i_m})=(x_1\wedge x_{i_2})\vee\ldots\vee (x_1\wedge x_{i_m})=x.\]
Also
\[x_1\wedge(x_2\vee\ldots\vee x_{k+1})=x.\]
Therefore
\[x_2\vee\ldots\vee x_{k+1}=x_2\vee\ldots\vee x_{k+1}\vee x_{i_1}\vee\ldots\vee x_{i_m}.\]
which is not possible by the induction hypothesis. Thus the statement is true for $\ell=k+1$. This completes the proof.
\end{proof}

We now use Proposition ~\ref{prop:p-g}, Theorem~\ref{thm:alp_finitecase} and Proposition ~\ref{prop:distlatticeproperty} to prove Theorem~\ref{thm:cm_lattice}.
\begin{proof}[Proof of Theorem~\ref{thm:cm_lattice}]
We start with the proof of the first part. Let $f$ be any c.m.\ function on $\LL$. Fix $\alp \ge d_{max}-1$. We want to prove that $f^\alp$ is c.m. By Proposition~\ref{prop:p-g} it is enough to show that there exists a non-negative function $p$ on $\LL$ such that 
\begin{align}\label{eq:p-falp} f^\alp(x)=\sum_{y\ge x} p(y).\end{align}
For the maximum element $m$ of $\LL$ define $p(m)=f^\alp(m)$. Inductively, having defined $p(y)$ for all $y>x$, define \[p(x):=f^\alp(x)-\sum_{y>x} p(y).\]
By the definition, $p$ satisfies~\eqref{eq:p-falp}. We now show that $p$ is non-negative. For any $x\in \LL$ we have
\ba
p(x)&=f^\alp(x)-\sum_{y>x} p(y)\\
&=f^\alp(x)-\sum_{i=1}^n f^\alp(x\vee x_i) + \sum_{i\le i<j\le n}f^\alp(x\vee x_i\vee x_j) -\sum_{i\le i<j<k\le n}f^\alp(x\vee x_i\vee x_j\vee x_k) + \ldots\\
&=\Delta_{x_n}\ldots\Delta_{x_1}f^\alp(x).
\ea
where $x_1,\ldots,x_n$ are all the covering elements of $x$.
% that is, 
% %
% \[
% \{y\in\LL: y>x, \nexists z\in\LL \text{ with } x < z < y\}= \{x_1,\ldots,x_n\}.\]
%
Now consider the Boolean lattice $\tilde{\LL}$ of the subsets of $[n]$ and the function $g$ on $\tilde{\LL}$ defined by
\[g(A):=f\left(x\vee\left(\bigvee_{i\in A}x_i\right)\right),\quad A\subset [n].\]
Then $g$ is c.m.\ Indeed, if $A, A_1,\ldots, A_k\subset [n]$, then
\ba
\Delta_{A_k}\ldots\Delta_{A_1}g(A)&=\sum_{J\subset[k]}(-1)^{|J|}g\left(A\cup\left(\bigcup_{j\in J}A_j\right)\right)\\
&= \sum_{J\subset[k]} (-1)^{|J|} f\left(x\vee x_{A}\vee \left(\bigvee_{j\in J}x_{A_j}\right)\right)\\
&=\Delta_{x_{A_k}}\ldots\Delta_{x_{A_1}}f(x\vee x_{A}) \ge 0,
\ea
where $x_B:=\bigvee_{i\in B}x_i \in\LL$ for any $B\subset [n]$.

Note that by Choquet's theorem (see Theorem $1.1.29$~\cite{molchanov2017}), up to a constant multiple, $g$ is void functional of a random subset of $[n]$. 
Since $n\le d_{max}$, we have $\alp \ge n-1$. Hence by Theorem~\ref{thm:alp_finitecase}, it follows easily that the function $g^\alp$ is c.m.\ Thus we have
\ba 
p(x)=\Delta_{x_n}\ldots\Delta_{x_1}f^\alp(x)=\Delta_{\{n\}}\ldots \Delta_{\{1\}}g^\alp(\emptyset) \ge 0.
\ea
This completes the proof of the first part.

We now prove the second part. We prove that there is a c.m.\ function $f$ on $\LL$ such that $f^\alp$ is not c.m for any non-integer $\alp<d_{max}-1$. We use the example of uniform random singleton set from Theorem ~\ref{thm:alp_finitecase} to define such a function $f$. 

Let $\XX$ be the random set defined by $\P\{\XX=\{i\}\}=1/d_{max}$ for $i=1,\dots d_{max}$ with $E=[d_{max}]$, that is, $\XX$ is the uniform singleton on $[d_{max}]$. Then the void functional $V_\XX$ of $\XX$ is a c.m.\ function on the Boolean lattice of the subsets of $[d_{max}]$. The function $V_\XX$ has the property that $V_\XX(A)$ depends only on $|A|$ and $d_{max}$. Also, by Theorem~\ref{thm:alp_finitecase} we have that for any non-integer $\alp<d_{max}-1$, the function $V_\XX^\alp$ is not c.m.\

Suppose $x\in\LL$ such that $d_x=d_{max}$. Let $x_1,\ldots,x_{d_{max}}$ be all the covering elements of $x$.
% %
% \[
% \{y\in\LL: y>x, \nexists z\in\LL \text{ with } x < z < y\}= \{x_1,\ldots,x_{d_{max}}\}.\]
% %
First, we define a function $g$ on the Boolean sub-lattice $\LL_x:=\lbrace  y\in\LL:y=x\vee \left( \bigvee_{j\in J} x_j \right), J\subset [d_{max}] \rbrace$ of $\LL$ such that $g$ is c.m.\ but $g^\alp$ is not c.m.\ for any non-integer $\alp<d_{max}-1$. Define $g\left(x\vee \left( \bigvee_{j\in J} x_j \right)\right):= V_{\XX}(J)$ and for any $J \subset [d_{max}]$. The function $g$ is well defined due to Proposition~\ref{prop:distlatticeproperty}. We first prove that $g$ is c.m.\ Indeed, for $y_0, y_1,\ldots, y_k\in\LL_x$ we have
\ba
\Delta_{y_k}\ldots\Delta_{y_1}g(y_0)&= \sum_{J\subset [k]} (-1)^{|J|} g\left(y_0\vee\left(\bigvee_{j\in J} y_j\right)\right)\\
&=\Delta_{A_k}\ldots\Delta_{A_1}V_\XX(A_0) \ge 0,
\ea
where $y_i=x\vee \left(\bigvee_{j\in A_i}x_j\right)$ for $i=0,1,\ldots,k$. Now fix any non-integer $\alp < d_{max}-1$. Since $V_{\XX}^\alp$ is not c.m.\, we have 
\[\Delta_{A_k}\ldots\Delta_{A_1}V_\XX^\alp(A) <0\]
for some $A, A_1, \ldots, A_k\subset [d_{max}]$. But
\[\Delta_{A_k}\ldots\Delta_{A_1}V_{\XX}^\alp(A)=\Delta_{x_{A_k}}\ldots \Delta_{x_{A_1}}g^\alp(x_A)\]
where $x_B:=x\vee\left(\bigvee_{i\in B}x_i\right) \in\LL_x$ for any $B\subset [d_{max}]$. Thus we have that $g^\alp$ is not c.m.\ for any non-integer $\alp <d_{max}-1$.

Now we extend $g$ to a function $f$ on $\LL$ such that $f$ is c.m.\ but $f^\alpha$ is not c.m.\ for any non-integer $\alp <d_{max}-1$. By Proposition~\ref{prop:p-g} there exists $p:\LL_x \to [0, \infty)$ such that $g(z)=\sum_{y\in\LL_x,\, y\ge z} p(y)$ for all $z\in\LL_x$. We extend $p$ to $\LL$ by defining it to be zero outside $\LL_x$ and define the function $f$ by 
\[f(z):=\sum_{y\ge z} p(y), \qquad z\in\LL.\]
Then by Proposition~\ref{prop:p-g} the function $f$ is c.m.\ Note that by definition $f|_{\LL_x}=g$. Since $g^\alp$ is not c.m.\ for any non-integer $\alp <d_{max}-1$, $f^\alp$ is not c.m.\ This completes the proof. 
\end{proof}

\section{Proofs of Theorem ~\ref{thm:dist_void} and ~\ref{thm:cm_approx_lattice}}\label{sec: proof of dist}
% In this section $C$ is a generic positive constant which may be different in each appearance.
\begin{proof}[Proof of Theorem~\ref{thm:dist_void}]
First we prove the upper bound.
% , that is 
% $$ \sup\limits_{\XX\in \DD_m}\inf\limits_{\YY\in \DD_{\infty}} d_V(\XX,\YY) \le \frac{2}{me^2}.$$
Fix $\XX_m\in\DD_m$. Let $\XX$ (may depend on $m$) be such that $\XX_m$ is the union of $m$ many i.i.d.\ copies of $\XX$. Let $\YY$ be the union of $N$ many i.i.d.\ copies of $\XX$, where $N\sim \mbox{Poi}(m)$. Then $\YY\in\DD_{\infty}$. For any $K\in\KK$ we have,
 \begin{align*}\label{ndiv_approx_contra}
    \left|V_{\XX_m}(K)-V_{\YY}(K)\right|&=\left|V_{\XX}^m(K)-e^{m\left(V_{\XX}(K)-1\right)}\right| \\
    &\leq \sup_{0\leq t\leq 1}\left| t^m-e^{m(t-1)}\right|.
\end{align*} 
One can check that the supremum in the above inequality occurs for $t_m\in(0,1)$ satisfying $\frac{-\log t_m}{1-t_m}=\frac{m}{m-1}$. It can be checked that $t_m=1-\frac{2}{m}+O(1/m^2)$. Hence for any $m\geq 1$,
\begin{align*}
    \sup_{0\leq t\leq 1}\left| t^m-e^{m(t-1)}\right|=t_m^{m-1}-t_m^m\leq \frac{c_1}{m}
\end{align*}
for some constant $0<c_1<\infty$. Thus for any $\XX_m\in\DD_m$ we proved that $\inf\limits_{\YY\in \DD_{\infty}} d_V(\XX_m,\YY) \le \frac{c_1}{m}$.
This completes the proof of the upper bound. Observe that $\lim\limits_{m\rightarrow \infty} m(t_m^{m-1}-t_m^m)=\frac{2}{e^2}$. Therefore $\limsup_m m\psi_m\leq 2/e^2$.

We now prove the lower bound. Since $|E|>1$, let $a,b\in E$ with $a\neq b$.  Consider the random set $\XX$ with $$\P\{\XX=\emptyset\}=1-\frac{1}{m},\quad   \P\{\XX=\{a\}\}=\frac{1}{2m},\quad   \P\{\XX=\{b\}\}=\frac{1}{2m}.$$ Let $\XX_m$ be the union of $m$ many i.i.d.\ copies of $\XX$. By definition, we have 
\begin{align*}
V_\XX(\phi)=1,\quad V_\XX(\{a\})=1-\frac{1}{2m}, \quad V_\XX(\{b\})=1-\frac{1}{2m}, \quad V_\XX(\{a,b\})=1-\frac{1}{m}\\
V_{\XX_m}(\phi)=1,\quad V_{\XX_m}(\{a\})=\left(1-\frac{1}{2m}\right)^m, \quad V_{\XX_m}(\{b\})=\left(1-\frac{1}{2m}\right)^m, \quad V_{\XX_m}(\{a,b\})=\left(1-\frac{1}{m}\right)^m.
\end{align*}

Let $C=1/(4\sqrt{e}(2+\sqrt{e}))$. If we show that $\liminf_m m\inf\limits_{\YY\in \DD_{\infty}} d_V(\XX_m,\YY)\geq C$, then the lower bound in \eqref{eq:n_div_approx} is proved. If possible, suppose it is not true. Then by going to subsequence we can get $\YY_m\in\DD_{\infty}$ such that $m d_V(\XX_m,\YY_m) \rightarrow \rho$ for some $0\leq\rho<C$, as $m\to\infty$. In that case, we must have $V_{\YY_m}(A)=V_{\XX_m}(A)+\frac{\rho_A}{m}+o(1/m)$, with $|\rho_A|\le \rho$ for all $A\in\{\phi,\{a\},\{b\},\{a,b\}\}.$ 
%Let $\YY$ be any IDRACS. %with $\P(\YY=\emptyset)=p_0,\  \P(\YY=\{1\})=p_1,\  \P(\YY=\{2\})=p_2,\ \P(\YY=\{1,2\})=p_{12}$. 
But for any $\YY$ to be infinitely divisible random set, we must have 
\begin{align*}
    V_\YY(\emptyset)^{\alpha}-V_\YY(\{a\})^{\alpha}-V_\YY(\{b\})^{\alpha}+V_\YY(\{a,b\})^{\alpha}\geq 0, \quad \forall \alpha\geq 0.
\end{align*}

This forces $V_\YY(\{a,b\})\geq V_\YY(\{a\}) V_\YY(\{b\})$, as the derivative of the above function at $\alpha=0$ has to be non-negative. So, for $\YY_m$ we must have
\begin{align*}
    \left(1-\frac{1}{m}\right)^m+\frac{\rho_{\{a,b\}}}{m}+o(1/m)\geq \left(\left(1-\frac{1}{2m}\right)^m+\frac{\rho_{\{a\}}}{m}+o(1/m)\right)\left(\left(1-\frac{1}{2m}\right)^m+\frac{\rho_{\{b\}}}{m}+o(1/m)\right).
\end{align*}
Since $\left(1-\frac{1}{m}\right)^{m}-\left(1-\frac{1}{2m}\right)^{2m}\geq 4/em$ for large $m$, the above implies $$\rho_{\{a,b\}}-\left(\rho_{\{a\}}+\rho_{\{b\}}\right)e^{-\frac12} -\frac1{4e} \ge 0 .$$ But
\begin{align*}
\rho_{\{a,b\}}-\left(\rho_{\{a\}}+\rho_{\{b\}}\right)e^{-\frac12} -\frac1{4e} \le \rho+2\rho e^{-\frac12}-\frac1{4e} < C\left(1+\frac2{\sqrt{e}}\right)-\frac1{4e}=0,
\end{align*}
which is a contradiction.
Hence we have proved the lower bound in \eqref{eq:n_div_approx}. From the proof it follows that $\liminf_m m\psi_m\geq C$.

\end{proof}

% {\color{red} Try for direct argument avoiding contradiction method: Let $n\in \N$ and $\XX_n\in\DD_n$. Then there is a random closed set $\XX$ such that $\XX_n$ is the union of $n$ many i.i.d. copies $\XX$. Let $M\sim$ Poi$(n)$ and $\XX_M$ be the union of $M$ many i.i.d. copies of $\XX$. Then $\XX_M\in\DD$. We show that there exists $c_2>0$ such that $d_V(\XX_n,\XX_M)\leq c_2/n$. This will complete the proof of the upper bound. We have
% \begin{align*}
%  d_V(\XX_n,\XX_M)&=\sup\limits_{K\in\KK}\left|V_{\XX_n}(K)-V_{\XX_M}(K)\right|\\
%  &= \sup\limits_{K\in\KK}\left|\exp\left(-n\log\left(1/V_{\XX}(K)\right)\right)-\exp\left(-n\left(1-V_{\XX}(K)\right)\right)\right|
% %&\leq n \sup\limits_{K\in\KK}\left|\log\left(1/V_{\XX}(K)\right)-\left(1-V_{\XX}(K)\right)\right|
% \end{align*}
% \textbf{Case 1:} Suppose $\sup_{K\in\KK} V_{\XX}(K)=C<1$. Then we have
% \begin{align*}
%  d_V(\XX_n,\XX_M)&=\sup\limits_{K\in\KK}\left|\exp\left(-n\log\left(1/V_{\XX}(K)\right)\right)-\exp\left(-n\left(1-V_{\XX}(K)\right)\right)\right|\\
%  & \le \exp\left(-n\log\left(1/C\right)\right)-\exp\left(-n\left(1-C\right)\right)\\
%  &\le \frac{c_2}{n}
%  \end{align*}
% for some $c_2>0$.\\
% \textbf{Case 2:} Suppose $\sup_{K\in\KK} V_{\XX}(K)=1$.
% Let $1-V_{\XX}(K)\le C/\sqrt{n}$ for some $C>0$.
% }

The proof of Theorem ~\ref{thm:cm_approx_lattice} is similar to the above proof of Theorem ~\ref{thm:dist_void}. Here we give a brief sketch of the proof. For this proof we need the following lemma, which follows from Proposition ~\ref{prop:p-g}.

\begin{lemma}\label{Corollary f-g}
Given a function $f$ which is c.m.\ on $\mathbb{S}$, a finite sub-lattice of lattice $\mathbb{L}$ (not necessarily finite), there exists a c.m.\ function $g$ on $\LL$ such that $g$ is an extension of $f$.
\end{lemma}

\begin{proof}
Since $f$ is c.m.\ on $\S$, by  Proposition ~\ref{prop:p-g} there exists $p: \S\rightarrow[0,\infty)$ such that  $f(x)=\sum_{y\geq x}p(y)$. We extend $p$ to $\LL$ by defining $p(x)=0$ for $x\notin\mathbb{S}$. Define $g(x)=\sum_{y\geq x}p(y)$ for $x\in \mathbb{L}$. As $p$ is non-zero on only finitely many lattice points, the sum is well defined. Note that on any finite sub-lattice $\mathbb{K}$, the function $g$ is c.m.\ due to Proposition ~\ref{prop:p-g}. This gives that $g$ is c.m.\ on $\mathbb{L}$.
\end{proof}

\begin{proof}[Proof of Theorem~\ref{thm:cm_approx_lattice} ]
As $\mathbb{L}$ is a lattice which is not a chain, there exists a square sub-lattice $\mathbb{M}$ of four distinct elements, say, $a,b,c,d$ with $a=b\wedge c$ and $d=b\vee c$. Consider $f$ defined on $\mathbb{M}$ with $f(a)=1, f(b)= 1-\frac{1}{2m}, f(c)= 1-\frac{1}{2m}, f(d)=1-\frac{1}{m}$. It is easy to check that $f$ is c.m.\ on $\mathbb{M}$. Using Lemma~\ref{Corollary f-g}, we can extend $f$ to $g$ which is c.m.\ on $\mathbb{L}$. Note that $g^m$ is $m$-divisible. Similarly as in the lower bound proof of Theorem ~\ref{thm:dist_void}, we can prove that $\tau_m\geq c_2/m$, using $g^m$ in place of $V_{\XX_m}$. For the upper bound of $\tau_m$, we follow the upper bound proof of Theorem ~\ref{thm:dist_void}. Here we use $g=\exp(-m(1-f^{1/m}))$ to approximate any $m$-divisible function $f$ taking values in $[0,1]$. One can show that for any $C>0$, the function $\exp(-C(1-f^{1/m}))$ is c.m.\ Hence $g$ is an infinitely divisible function. Similar argument as in the upper bound proof of Theorem ~\ref{thm:dist_void} shows that $\tau_m\leq c_1/m$. 
\end{proof}

\section{Proofs of Theorem~\ref{thm:multipleintervals} and~\ref{thm:oneinterval}}\label{sec: proof of one, multiple}

In this section, we first prove Theorem~\ref{thm:oneinterval}. To prove Theorem~\ref{thm:oneinterval} we use the following lemma from~\cite{zhang1998} on Schur convexity. For $x=(x_1, \ldots, x_n)$ and $y=(y_1, \ldots, y_n)$ we say that $x$ is majorized by $y$, and write $x\prec y$, if
\ba
\sum_{i=1}^n x_{[i]} = \sum_{i=1}^n y_{[i]}
\quad \mbox{ and }\quad \sum_{i=1}^k x_{[i]} \le \sum_{i=1}^k y_{[i]}, \quad k=1,\ldots, n-1
 \ea
where $x_{[i]}$ denotes the $i$-th largest component in $x$. A function $f:A\to \R$ is Schur convex on $A\subset \R^n$ if $f(x)\le f(y)$ for each $x, y\in A$ with $x\prec y$ holds. $f$ is strictly Schur convex on $A$ if $f(x)<f(y)$ whenever $x\prec y$ and $x$ is not a permutation of $y$.
\begin{lemma}[Lemma 1.5~\cite{zhang1998}]\label{lem:schurconvexity}
Let $f$ be a symmetric function in $x_1,\ldots,x_n$ on $A\subset \R^n$ that has continuous partial derivatives. Suppose, $A$ satisfies the following
\begin{enumerate}
    \item $A$ is symmetric, i.e., if $x=(x_1,\ldots, x_n)\in A$, then $Px\in A$ for any $n\times n$ permutation matrix $P$.
    \item $A$ is convex and has a non-empty interior.
\end{enumerate}
Then $f$ is strictly Schur convex if and only if
\ba
    (x_i-x_j)\left(\frac{\partial f}{\partial x_i} - \frac{\partial f}{\partial x_j}\right) >0
\ea
on $A$ for $x_i\neq x_j$, $1\le i,j \le n$. Since $f$ is symmetric, the above condition can be reduced to
\begin{align}\label{eq:schurconvexitycondition}
    (x_1-x_2)\left(\frac{\partial f}{\partial x_1} - \frac{\partial f}{\partial x_2}\right) >0
\end{align}
for $x_1\neq x_2$.
\end{lemma}

First, we prove the following lemma using Lemma~\ref{lem:schurconvexity}.
\begin{lemma}\label{lem:1singleton}
For any $n\ge 2$ and any $\alp\in (n-1,n)$, consider the function
\[f_{n,\alp}(x_1,\ldots, x_n):= 1+ \sum_{k=1}^n \sum_{B\subset [n], |B|=k}\left[(-1)^{n+1-k} \left(\sum_{i\in B}x_i \right)^\alp + (-1)^k \left(1-\sum_{i\in B}x_i\right)^\alp \right]
\]
defined on 
\[A_n:=\{(x_1,\ldots, x_n): 0<x_1,\ldots,x_n<1, \sum_{i=1}^n x_i <1\}.\]
Then $f_{n,\alp}$ is strictly Schur convex on $A_n$.
\end{lemma}
\begin{proof}
We prove this lemma by induction on $n$. Since $f_{n,\alp}$ is symmetric, to prove strict Schur convexity we show that the condition~\eqref{eq:schurconvexitycondition} holds.

Consider $n=2$. Let $\alp\in(1,2)$. We have
\[f_{2,\alp}(x_1, x_2)=1+(x_1^\alp -(1-x_1)^\alp)+ (x_2^\alp -(1-x_2)^\alp) - ((x_1+x_2)^\alp -(1-x_1-x_2)^\alp)\]
defined on $A_2:=\{(x_1,x_2):0<x_1, x_2<1, x_1+x_2<1\}$. We show that for $(x_1,x_2)\in A_2$ with $x_1\neq x_2$
\begin{align}\label{eq:basecaseschurconvexity}
(x_1-x_2)\left(\frac{\partial f_{2,\alp}}{\partial x_1} - \frac{\partial f_{2,\alp}}{\partial x_2}\right)(x_1,x_2) >0.\end{align}
Then by Lemma~\ref{lem:schurconvexity}, $f_{2,\alp}$ will be strictly Schur convex. We have 
\begin{align}\label{eq:basecaseschurconvexity2}
\left(\frac{\partial f_{2,\alp}}{\partial x_1} - \frac{\partial f_{2,\alp}}{\partial x_2}\right)(x_1,x_2)&= \alp\left((x_1^{\alp-1} +(1-x_1)^{\alp-1})- (x_2^{\alp-1} +(1-x_2)^{\alp-1})\right) \nonumber\\
&=\alp \left( h_{2,\alp}(x_1) -h_{2,\alp}(x_2)\right)
\end{align}
where the function $h_{2,\alp}$ is defined by
\[h_{2,\alp}(y):=y^{\alp-1} +(1-y)^{\alp-1}, \quad 0<y<1.\]
Observe that $h_{2,\alp}$ is symmetric, i.e., $h_{2,\alp}(y)=h_{2,\alp}(1-y)$. Also, we have
\[h_{2,\alp}^{'}(y)=(\alp-1) \left(y^{\alp-2} -(1-y)^{\alp-2}\right) >0 \]
for all $y\in(0,1/2)$. Therefore $h_{2,\alp}$ is strictly increasing in $(0,1/2)$. Now to show that~\eqref{eq:basecaseschurconvexity} holds, we assume $x_1>x_2$. The argument is similar if $x_1<x_2$. Since $x_1+x_2<1$, we have $x_2 <1/2$. If $x_1 <1/2$, then $h_{2,\alp}(x_2) < h_{2,\alp}(x_1)$, as $h_{2,\alp}$ is strictly increasing in $(0,1/2)$. Again, if $x_1 \ge 1/2$, then we have $x_2<1-x_1\le 1/2$ and hence $h_{2,\alp}(x_2) < h_{2,\alp}(1-x_1)=h_{2,\alp}(x_1)$. Therefore from~\eqref{eq:basecaseschurconvexity2} we conclude that~\eqref{eq:basecaseschurconvexity} holds. Thus we proved that the statement of the lemma is true for $n=2$.

Suppose the statement is true for $n=2,3,\ldots, m$. We prove that it is true for $n=m+1$. Fix $\alp\in(m, m+1)$. Consider the function
\[f_{m+1,\alp}(x_1,\ldots, x_{m+1}):= 1+ \sum_{k=1}^{m+1} \sum_{B\subset [m+1], |B|=k}\left[(-1)^{m+2-k} \left(\sum_{i\in B}x_i \right)^\alp + (-1)^k \left(1-\sum_{i\in B}x_i\right)^\alp \right]
\]
defined on 
\[A_{m+1}:=\{(x_1,\ldots, x_{m+1}): 0<x_1,\ldots,x_{m+1}<1, \sum_{i=1}^{m+1} x_i <1\}.\]
By Lemma~\ref{lem:schurconvexity}, it is enough to prove that for any $x=(x_1,\ldots, x_{m+1})\in A_{m+1}$ with $x_1\neq x_2$
\begin{align}\label{eq:inductionstepschurconvexity}
(x_1-x_2)\left(\frac{\partial f_{m+1,\alp}}{\partial x_1} - \frac{\partial f_{m+1,\alp}}{\partial x_2}\right)(x) >0.\end{align}
Fix $x=(x_1,\ldots, x_{m+1})\in A_{m+1}$ with $x_1\neq x_2$. We have
\ba
\left(\frac{\partial f_{m+1,\alp}}{\partial x_1} - \frac{\partial f_{m+1,\alp}}{\partial x_2}\right)(x) &= \alp \sum_{k=1}^{m+1} \sum_{B\subset [m+1], |B|=k, 1\in B, 2\notin B}\left[(-1)^{m+2-k} \left(\sum_{i\in B}x_i \right)^{\alp-1} + (-1)^{k+1} \left(1-\sum_{i\in B}x_i\right)^{\alp-1} \right]\\
& \quad - \alp \sum_{k=1}^{m+1} \sum_{B\subset [m+1], |B|=k, 1\notin B, 2\in B}\left[(-1)^{m+2-k} \left(\sum_{i\in B}x_i \right)^{\alp-1} + (-1)^{k+1} \left(1-\sum_{i\in B}x_i\right)^{\alp-1} \right].
\ea
Define
\[h_{m+1,\alp, x_3,\ldots,x_{m+1}} (y):=\sum_{k=0}^{m} \sum_{B\subset \{3,\ldots, m+1\}, |B|=k}\left[(-1)^{m+1-k} \left(y+\sum_{i\in B}x_i \right)^{\alp-1} + (-1)^{k+2} \left(1-y-\sum_{i\in B}x_i\right)^{\alp-1} \right]\]
for $0<y<1-\sum_{i=3}^{m+1}x_i$. Note that by definition
\begin{align}\label{eq:inductionstepschurconvexity2}
\left(\frac{\partial f_{m+1,\alp}}{\partial x_1} - \frac{\partial f_{m+1,\alp}}{\partial x_2}\right)(x) = \alp \left[ h_{m+1,\alp, x_3,\ldots,x_{m+1}}(x_1) - h_{m+1,\alp, x_3,\ldots,x_{m+1}}(x_2)\right].
\end{align}
We now make the following claim whose proof uses the induction hypothesis and is postponed till the end of the proof.  
\begin{claim}\label{claim:inductionstepschurconvexity}
\,
\benu
\item $h_{m+1,\alp, x_3,\ldots,x_{m+1}}$ is symmetric, i.e., $h_{m+1,\alp, x_3,\ldots,x_{m+1}}(y)=h_{m+1,\alp, x_3,\ldots,x_{m+1}}(1-\sum_{i=3}^{m+1}x_i-y)$.
\item $h_{m+1,\alp, x_3,\ldots,x_{m+1}}$ is strictly increasing in $(0,\frac12(1-\sum_{i=3}^{m+1}x_i))$.
\eenu
\end{claim}
We use the above claim and complete the proof. Without loss of generality, assume $x_1>x_2$. The argument is similar if $x_1<x_2$. If $x_1 < \frac12(1-\sum_{i=3}^{m+1}x_i)$, then $h_{m+1,\alp, x_3,\ldots,x_{m+1}}(x_2) < h_{m+1,\alp, x_3,\ldots,x_{m+1}}(x_1)$. Again, if $x_1 \ge \frac12(1-\sum_{i=3}^{m+1}x_i)$, then $x_2 < 1-x_1-\sum_{i=3}^{m+1}x_i \le  \frac12(1-\sum_{i=3}^{m+1}x_i)$ and hence $h_{m+1,\alp, x_3,\ldots,x_{m+1}}(x_2) < h_{m+1,\alp, x_3,\ldots,x_{m+1}}(1-x_1-\sum_{i=3}^{m+1}x_i ) =h_{m+1,\alp, x_3,\ldots,x_{m+1}}(x_1)$. Therefore, from~\eqref{eq:inductionstepschurconvexity2} we conclude that~\eqref{eq:inductionstepschurconvexity} holds. Thus we proved that the statement of the lemma is true for $n=m+1$.

We now prove Claim~\ref{claim:inductionstepschurconvexity} using the induction hypothesis.
\begin{proof}[Proof of Claim~\ref{claim:inductionstepschurconvexity}]
\,
\benu
\item For any $B\subset \{3,\ldots,m+1\}$ with $|B|=k$, we have
\ba
(-1)^{m+1-k} \left(\left(1-\sum_{i=3}^{m+1}x_i-y\right)+\sum_{i\in B}x_i \right)^{\alp-1}&= (-1)^{m+1-k} \left(1-y-\sum_{i\in B^c}x_i\right)^{\alp-1}\\
&= (-1)^{|B^c|+2} \left(1-y-\sum_{i\in B^c}x_i\right)^{\alp-1}
\ea
and 
\ba
(-1)^{k+2} \left(1-\left(1-\sum_{i=3}^{m+1}x_i-y\right)-\sum_{i\in B}x_i \right)^{\alp-1}&= (-1)^{m+1-|B^c|} \left(y+\sum_{i\in B^c}x_i \right)^{\alp-1}.
\ea
These show that $h_{m+1,\alp, x_3,\ldots,x_{m+1}}(y)=h_{m+1,\alp, x_3,\ldots,x_{m+1}}(1-\sum_{i=3}^{m+1}x_i-y)$.
\item We show that $h_{m+1,\alp, x_3,\ldots,x_{m+1}}^{'}(y) >0$ for all $0<y<\frac12(1-\sum_{i=3}^{m+1}x_i)$. We have
\ba
h_{m+1,\alp, x_3,\ldots,x_{m+1}}^{'}(y) &= (\alp-1)\sum_{k=0}^{m} \sum_{B\subset \{3,\ldots, m+1\}, |B|=k}\left[(-1)^{m+1-k} \left(y+\sum_{i\in B}x_i \right)^{\alp-2} + (-1)^{k+3} \left(1-y-\sum_{i\in B}x_i\right)^{\alp-2} \right]\\
&= (\alp-1)\sum_{k=0}^{m} \sum_{B\subset \{3,\ldots, m\}, |B|=k}\left[(-1)^{m-k} \left(y+x_{m+1}+\sum_{i\in B}x_i \right)^{\alp-2} + (-1)^{k+4} \left(1-y-x_{m+1}-\sum_{i\in B}x_i\right)^{\alp-2} \right]\\
&\quad + (\alp-1)\sum_{k=0}^{m} \sum_{B\subset \{3,\ldots, m\}, |B|=k}\left[(-1)^{m+1-k} \left(y+\sum_{i\in B}x_i \right)^{\alp-2} + (-1)^{k+3} \left(1-y-\sum_{i\in B}x_i\right)^{\alp-2} \right].
\ea
Now consider the function
\[f_{m,\alp-1}(z_1,\ldots, z_m):= 1+ \sum_{k=1}^{m} \sum_{B\subset [m], |B|=k}\left[(-1)^{m+1-k} \left(\sum_{i\in B}z_i \right)^{\alp-1} + (-1)^k \left(1-\sum_{i\in B}z_i\right)^{\alp-1} \right]
\]
defined on 
\[A_{m}:=\{(z_1,\ldots, z_{m}): 0<z_1,\ldots,z_{m}<1, \sum_{i=1}^{m} z_i <1\}.\]
By induction hypothesis $f_{m, \alp-1}$ is strictly Schur convex on $A_m$. Therefore by Lemma~\ref{lem:schurconvexity} we have 
\[\left(\frac{\partial f_{m,\alp-1}}{\partial z_1} - \frac{\partial f_{m,\alp-1}}{\partial z_2}\right)(z) >0\]
for any $z=(z_1,\ldots, z_{m})\in A_m$ with $z_1\neq z_2$. We compute
\ba 
\left(\frac{\partial f_{m,\alp-1}}{\partial z_1} - \frac{\partial f_{m,\alp-1}}{\partial z_2}\right)(z)&=(\alp-1) \sum_{k=1}^{m} \sum_{B\subset [m], |B|=k, 1\in B, 2\notin B}\left[(-1)^{m+1-k} \left(\sum_{i\in B}z_i \right)^{\alp-2} + (-1)^{k+1} \left(1-\sum_{i\in B}z_i\right)^{\alp-2} \right]\\
& \quad - (\alp-1) \sum_{k=1}^{m} \sum_{B\subset [m], |B|=k, 1\notin B, 2\in B}\left[(-1)^{m+1-k} \left(\sum_{i\in B}z_i \right)^{\alp-2} + (-1)^{k+1} \left(1-\sum_{i\in B}z_i\right)^{\alp-2} \right]\\
&= (\alp-1) \sum_{k=0}^{m} \sum_{B\subset \{3,\ldots,m\}, |B|=k}\left[(-1)^{m-k} \left(z_1+\sum_{i\in B}z_i \right)^{\alp-2} + (-1)^{k+2} \left(1-z_1-\sum_{i\in B}z_i\right)^{\alp-2} \right]\\
& \quad - (\alp-1) \sum_{k=0}^{m} \sum_{B\subset \{3,\ldots,m\}, |B|=k}\left[(-1)^{m-k} \left(z_2+\sum_{i\in B}z_i \right)^{\alp-2} + (-1)^{k+2} \left(1-z_2-\sum_{i\in B}z_i\right)^{\alp-2} \right].
\ea
Observe that for $0<y<\frac12(1-\sum_{i=3}^{m+1}x_i)$, we have $(y+x_{m+1},y,x_3,\ldots,x_m)\in A_m$ and 
\[h_{m+1,\alp, x_3,\ldots,x_{m+1}}^{'}(y)= \left(\frac{\partial f_{m,\alp-1}}{\partial z_1} - \frac{\partial f_{m,\alp-1}}{\partial z_2}\right)(y+x_{m+1},y,x_3,\ldots,x_m)>0.\]
Thus we proved that $h_{m+1,\alp, x_3,\ldots,x_{m+1}}$ is strictly increasing in $(0, \frac12(1-\sum_{i=3}^{m+1}x_i))$.
\eenu
The proof of the lemma is now complete.
\end{proof}
\end{proof}
Using Lemma~\ref{lem:1singleton} we prove the following.
\begin{lemma}\label{lem:2singleton}
For any $n\ge 2$ and $\alp\in(n-1,n)$ consider the function $f_{n,\alp}$ (from Lemma ~\ref{lem:1singleton}) defined on $\overline{A_n}=\{(x_1,\ldots,x_n): 0\le x_1,\ldots,x_n\le 1, \sum_{i=1}^n x_i\le 1\}$. Then $f_{n,\alp}\equiv 0$ on the boundary $\partial A_n$ of $A_n$. Also, $f_{n,\alp}(x)<0$ for any $x=(x_1,\ldots,x_n)\in A_n$.
\end{lemma}
\begin{proof}
First we prove that $f_{n,\alp}\equiv 0$ on $\partial A_n$. Recall that 
\[f_{n,\alp}(x_1,\ldots, x_n):= 1+ \sum_{k=1}^n \sum_{B\subset [n], |B|=k}\left[(-1)^{n+1-k} \left(\sum_{i\in B}x_i \right)^\alp + (-1)^k \left(1-\sum_{i\in B}x_i\right)^\alp \right]
\]
Let $x=(x_1,\ldots,x_n)\in \partial A_n$. First suppose that $x_i=0$ for some $i\in\{1,\ldots,n\}$. Then we have
\[ \sum_{k=1}^n \sum_{B\subset [n], |B|=k}(-1)^{n+1-k} \left(\sum_{i\in B}x_i \right)^\alp =0\]
and
\[ \sum_{k=1}^n \sum_{B\subset [n], |B|=k} (-1)^k \left(1-\sum_{i\in B}x_i\right)^\alp = -1.\]
So, $f_{n,\alp}(x)=0$. Next suppose $x_i>0$ for all $i$. Then we must have $\sum_{i=1}^n x_i=1$. In this case we have
\[\sum_{k=1}^n \sum_{B\subset [n], |B|=k}(-1)^{n+1-k} \left(\sum_{i\in B}x_i \right)^\alp = \sum_{k=1}^n \sum_{B\subset [n], |B|=k} (-1)^k \left(1-\sum_{i\in B}x_i\right)^\alp -1.\]
So, $f_{n,\alp}(x)=0$. Thus $f_{n,\alp}\equiv 0$ on $\partial A_n$.

Now fix $x=(x_1,\ldots,x_n)\in A_n$. We want to show that $f_{n,\alp}(x)<0$. Since $f_{n,\alp}$ is symmetric, without loss of generality, we can assume $x_1\le \ldots \le x_n$. If possible, suppose $f_{n,\alp}(x)\ge 0$. Observe that $(x_1/\ell, x_2, \ldots, x_{n-1}, x_n+x_1(1-1/\ell)) \prec (x_1/(\ell+1), x_2, \ldots, x_{n-1}, x_n+x_1(1-1/(\ell+1)))$ for any $\ell \ge 1$. Therefore, by Lemma~\ref{lem:1singleton} we have 
\[f_{n,\alp}\left(\frac{x_1}{\ell}, x_2, \ldots, x_{n-1}, x_n+x_1\left(1-\frac{1}{\ell}\right)\right) < f_{n,\alp}\left(\frac{x_1}{\ell+1}, x_2, \ldots, x_{n-1}, x_n+x_1\left(1-\frac{1}{\ell+1}\right)\right).\]
But this is not possible as $(x_1/\ell, x_2, \ldots, x_{n-1}, x_n+x_1(1-1/\ell)) \overset{\ell \to \infty}{\to} (0, x_2, \ldots, x_{n-1}, x_n+x_1) \in\partial A_n$. So, we must have $f_{n,\alp}(x)<0$.
\end{proof}
We now use Lemma~\ref{lem:2singleton} to prove Theorem~\ref{thm:oneinterval}.
\begin{proof}[Proof of Theorem~\ref{thm:oneinterval}]
 We know that $\XX_\alp$ exists if $\alp$ is an integer or if $\alp\ge n-1$ (see first part of  Theorem~\ref{thm:alp_finitecase}). We want to prove that if $\alp<n-1$ and $\alp$ is not an integer, then $\XX_\alp$ does not exist. Since the set $\{\alp\ge 0: \XX_\alp \text{ exists}\}$ is a semigroup under addition and $1$ belongs to the set, it is enough to prove that $\XX_\alp$ does not exists for $\alp\in(n-2, n-1)$.

Using M\"{o}bius inversion we have 
  \begin{align*}
&V_{\mathcal X_\alpha}(A^c)=\sum_{B\subset A}\mathbb P\{\mathcal X_\alpha=B\}\\
&\Rightarrow \mathbb P\{\mathcal X_\alpha=A\}= \sum_{B\subset A} (-1)^{|A|-|B|}V_{\mathcal X_\alpha}(B^c) = \sum_{B\subset A} (-1)^{|A|-|B|} \mathbb P\{\mathcal X \subset B\}^\alpha.
\end{align*}
 Thus it follows that $\XX_\alp$ exists if and only if for any $A\subset [n]$,
 \begin{align}\label{eq: existence of random set}
 q(A)= \sum_{B\subset A} (-1)^{|A|-|B|} \P\{\XX \subset B\}^\alp \ge 0.
\end{align} 

 % We know that $\XX_\alp$ exists if and only if for any $A\subset [n]$
 % % %
 % \[q(A)= \sum_{B\subset A} (-1)^{|A|-|B|} \P\{\XX \subset B\}^\alp \ge 0\]
 % %
 We prove that for any $\alp\in (n-2, n-1)$
 \begin{align}\label{eq: prop 15 eq}
g_n(\alp)= \sum_{B\subset [n]} (-1)^{n-|B|} \P\{\XX \subset B\}^\alp < 0. 
 \end{align}
For $n=1$, there is nothing to prove. $\XX_\alp$ exists for all $\alp \ge 0$. Suppose $n=2$. We have 
\[g_2(\alp)= 1-p_1^\alp -p_2^\alp = 1-p_1^\alp-(1-p_1)^\alp.\]
Note that $g_2(\alp)$ is an exponential polynomial in $\alp$ with one sign change. Therefore, it has at most one zero (by Proposition $3.2$ of~\cite{jain}). Also, $g_2(0)=-1<0$ and $g_2(1)=0$. Hence, $g_2(\alp)<0$ for $\alp\in (0, 1)$.

Now let $n\ge 3$ and $\alp\in (n-2, n-1)$. We have
\ba
g_n(\alp)&= \sum_{B\subset [n]} (-1)^{n-|B|} \P\{\XX \subset B\}^\alp\\
&= \sum_{k=1}^n \sum_{B\subset [n], |B|=k} (-1)^{n-k}\P\{\XX \subset B\}^\alp\\
&= \sum_{k=1}^n \sum_{B\subset [n], n\in B, |B|=k} (-1)^{n-k}\left(p_n +\sum_{i\in B\setminus\{n\}} p_i\right)^\alp + \sum_{k=1}^n \sum_{B\subset [n], n\notin B, |B|=k} (-1)^{n-k}\left(\sum_{i\in B} p_i\right)^\alp\\
&= \sum_{k=1}^{n} \sum_{B\subset [n], n\in B, |B|=k} (-1)^{|B^c|}\left(1-\sum_{i\in B^c} p_i\right)^\alp + \sum_{k=1}^n \sum_{B\subset [n], n\notin B, |B|=k} (-1)^{n-k}\left(\sum_{i\in B} p_i\right)^\alp\\
&= 1+ \sum_{k=1}^{n-1} \sum_{B\subset [n-1],|B|=k} (-1)^{k}\left(1-\sum_{i\in B} p_i\right)^\alp +\sum_{k=1}^{n-1} \sum_{B\subset [n-1],|B|=k}(-1)^{n-k}\left(\sum_{i\in B} p_i\right)^\alp \\
&= f_{n-1,\alp} (p_1,\ldots,p_{n-1}).
\ea
But since $0<p_1,\ldots,p_{n-1}<1$ and $\sum_{i=1}^{n-1} p_i <1$, from Lemma~\ref{lem:2singleton} we have $g_n(\alp)=f_{n-1,\alp} (p_1,\ldots,p_{n-1}) <0$. This completes the proof.
\end{proof}

To prove Theorem~\ref{thm:multipleintervals} we need the following result.
\begin{proposition}\label{prop:formultipleintervals}
 Fix $n\geq 1$ and let $p_1,\ldots,p_n$ be any positive real numbers. Define
 \[q(\alpha):=\sum_{A\subset[n]}(-1)^{n-|A|}\left(\sum_{i\in A}p_i\right)^{\alpha}.\]
Then $q(\alpha)$ is zero at all positive integers $\alpha\leq n-1$ and $q(\alpha)\geq 0$ for $\alpha>n-1$.
 \end{proposition}
\begin{proof}
The proof follows from the proof of Theorem~\ref{thm:oneinterval}. Indeed, define a random subset $\XX$ of $[n]$ such that $\P(\XX={i})=p_i/(\sum_i p_i)$. Then from ~\eqref{eq: prop 15 eq} and semigroup property of the set $\{\alpha\geq 0: \XX_{\alpha} \mbox{ exists}\}$ we conclude that $q(\alpha)<0$ for non-integer $\alpha<n-1$. Also $q(\alpha)\geq 0$ if $\alpha\in\N$ or $\alpha\geq n-1$ (since $\XX_{\alpha}$ exists in this case). Therefore $q(\alpha)=0$ for all positive integers $\alpha\leq m-1$.
% Fix any positive integer $\alpha\leq m-1$. After expanding out $q(\alpha)$ consider the terms of the form $p_{i_1}^{\alpha_1}p_{i_2}^{\alpha_2}\dots p_{i_k}^{\alpha_k}$ such that $\alpha_i >0$, $\sum_{i=1}^k \alpha_i=\alpha\leq m-1$ and $i_1,i_2,\dots i_k$ are distinct. Such terms appear as long as $A$ contains $i_1,\ldots, i_k$. For any $A\subset [m]$ containing $i_1,i_2,\dots i_k$, the coefficient of $p_{i_1}^{\alpha_1}p_{i_2}^{\alpha_2}\dots p_{i_k}^{\alpha_k}$ in $(\sum_{i\in A}p_i)^{\alpha}$ is $\alpha \choose {\alpha_1, \alpha_2, \dots ,\alpha_k}$. Hence the coefficient of $p_{i_1}^{\alpha_1}p_{i_2}^{\alpha_2}\dots p_{i_k}^{\alpha_k}$ in $q(\alpha)$ is ${\alpha \choose {\alpha_1, \alpha_2, \dots ,\alpha_k}} \left(1-{m-k\choose{1}} +{m-k\choose{2}}-{m-k\choose{3}} +\dots \right)$. Note that second factor is $0$, as $k<m$. Here $i_1,i_2\dots i_k, \alpha_1,\alpha_2,\dots \alpha_k$ are arbitrary. Hence $q(\alpha)=0$.
 
%  As $q(\alpha)$ has zeros at $\alpha=1,\dots, m-1$, and $q(\alpha)$ is an exponential polynomial with $m-1$ sign changes, it has no other zeros by Proposition $3.2$ of~\cite{jain}. Dividing the expression by $(p_1+\dots +p_m)^{\alpha}$ and letting $\alpha\rightarrow \infty$, it can be seen that $q(\alpha)>0$ for $\alpha>m-1$.
\end{proof}
We now prove Theorem~\ref{thm:multipleintervals}. Before we proceed we recall the necessary and sufficient condition \eqref{eq: existence of random set} for existence of $\XX_{\alp}$.
 
\begin{proof}[Proof of Theorem~\ref{thm:multipleintervals}]
We proceed by induction on $k$. First consider $k=2$. Define a random set $\XX^{(2)}$ such that $Q_{2}(\{i\})=\frac{1-\epsilon}{n}$ and $Q_{2}([n])=\epsilon$ for some $\epsilon >0$ which will be chosen later. Set $p_i=\frac{1-\epsilon}{n}$ for $i=1,\dots,n$.

We first show that for any choice of $\epsilon \ge 0$, $\XX_{\alpha}^{(2)}$ does not exist if $\alpha < n-3$ and $\alpha$ is non-integer. Indeed, if $\ell <\alp<\ell+1$ with $0\le \ell \le n-4$, then by Proposition ~\ref{prop:formultipleintervals} and Proposition $3.2$ of~\cite{jain}, we have for any $B\subset [n]$ with $|B|=\ell+2$
\[r_{2,B}(\alpha)=\sum_{A\subset B}(-1)^{|B|-|A|}\left(\sum_{i\in A} p_i \right)^{\alp}<0.\] 
 We now choose small enough $\epsilon >0$ so that $\XX_{\alpha}^{(2)}$ does not exist for some $\alp\in (n-2, n-1)$. We have  
\begin{align*}
r_{2,[n]}(\alpha)&=\sum_{A\subset [n]}(-1)^{n-|A|}\mathbb{P}\{X^{(2)}\subset A\}^{\alpha} \\
&= 1 + \sum_{A\subset [n], A\neq [n]}(-1)^{n-|A|}\mathbb{P}\{X^{(2)}\subset A\}^{\alpha} \\
&=1-(p_1+p_2+\dots p_n)^{\alpha}+\sum_{A\subset [n]}(-1)^{n-|A|}\left(\sum_{i\in A}p_i \right)^{\alpha} \\
&=1-(1-\epsilon)^{\alpha}+\sum_{A\subset [n]}(-1)^{n-|A|}\left(\sum_{i\in A}p_i \right)^{\alpha}.
\end{align*}

For $\epsilon=0$, by Proposition~\ref{prop:formultipleintervals} we have $r_{2,[n]}(\alpha)=0$ for $\alpha=n-2,n-1$ and $r_{2,[n]}(\alpha)<0$ for $n-2<\alpha<n-1$. Therefore for small enough $\epsilon>0$ one can have $r_{2,[n]}(\alpha)<0$ for some $n-2<\alpha<n-1$. 

Fix an $\epsilon>0$ such that $r_{2,[n]}(\alpha)<0$ for some $n-2<\alpha<n-1$. Note that the function $\alp \mapsto \sum_{A\subset [n]}(-1)^{n-|A|}(\sum_{i\in A }p_i)^{\alpha}$ is a continuous function which is zero at positive integers $\ell\le n-2$ (by Proposition~\ref{prop:formultipleintervals}) and there exists $\eta=\eta(\epsilon)>0$ such that 
\[1-(p_1+p_2+\dots p_n)^{\alpha} = 1-(1-\epsilon)^{\alp}>\eta \]
for all $\alpha\in[1/2,n-1]$. Hence for each positive integers $\ell\le n-2$ we have $r_{2,[n]}(\alpha)>0$ in an interval around $\ell$.

Now the existence of $\XX^{(2)}_\alp$ in an interval $[n-2,n-2+\delta)$ follows from the fact that $r_{2,A}(\alpha)\geq0$ for any $A\subset [n]$ with $|A|\leq n-1$ and $\alp\geq n-2$ (see Proposition ~\ref{prop:formultipleintervals}). This completes the base case $k=2$.

Suppose the statement of the theorem is true for $k=\ell$ with $\ell\le n-3$. We want to prove that the statement is true for $k=\ell+1$. We first define $X^{(\ell+1)}$ which is obtained by perturbing $X^{(\ell)}$. Define
\begin{align*}
Q_{\ell+1}(A)= \begin{cases}Q_{\ell}(A)-\epsilon c_1 \qquad\text{ if } |A|>n-\ell+1 \text{ or } |A|=1\\
\epsilon c_2 \qquad\qquad\qquad\text{if } |A|=n-\ell+1\\
0 \qquad\qquad\qquad \quad\text{otherwise}\end{cases} 
\end{align*}
where $\epsilon>0$ is small (will be chosen later) and $c_1,c_2$ are constants depending on $n$ and $\ell$ which are chosen in such a way that the sum of the probabilities is $1$. We set $s_i=Q_{\ell}(\{i\})-\epsilon c_1$ for $i=1,\dots, n$. By induction hypothesis, $s_i$ does not depend on $i$.

We first show that for any choice of $\epsilon \ge 0$, $\XX_{\alpha}^{(\ell+1)}$ does not exist if $\alpha < n-(\ell+1)-1$ and $\alpha$ is non-integer. Suppose $m <\alp<m+1$ with $0\le m \le n-(\ell+1)-2$. Then by Proposition ~\ref{prop:formultipleintervals} and Proposition $3.2$ of~\cite{jain}, we have for any $B\subset [n]$ with $|B|=m+2$
\[r_{\ell+1,B}(\alpha)=\sum_{A\subset B}(-1)^{|B|-|A|}\left(\sum_{i\in A} s_i \right)^{\alp}<0.\] 
Hence $\XX_{\alpha}^{(\ell+1)}$ does not exist.

 We now choose small enough $\epsilon >0$ so that $\XX_{\alpha}^{(\ell+1)}$ does not exist for some $\alpha_j\in (j,j+1)$ for all $n-(\ell+1)\leq j\leq n-2$. By induction hypothesis, $\XX_{\alpha}^{(\ell)}$ does not exist for some $\alpha\in(j,j+1), \forall n-\ell\leq j\leq n-2$. Since $\XX^{(\ell+1)}$ is defined by perturbing $\XX^{(\ell)}$, one can choose small enough $\epsilon>0$ to ensure that $\XX_{\alpha}^{(\ell+1)}$ does not exist for some $\alpha\in (j,j+1), \forall n-\ell\leq j\leq n-2$. It remains to show that $\XX_{\alpha}^{(\ell+1)}$ does not exist for some $\alpha\in (n-\ell-1,n-\ell)$. Let $B\subset [n]$ with $|B|=n-\ell+1$. We have
\begin{align}\label{eq:multipleintervalproof1}
r_{\ell+1,B}(\alpha)&=\sum_{A\subset B}(-1)^{|B|-|A|}\mathbb{P}\{X^{(\ell+1)}\subset A\}^{\alpha} \nonumber\\
&= \mathbb{P}\{X^{(\ell+1)}\subset B\}^{\alpha}+ \sum_{A\subset B, A\neq B}(-1)^{|B|-|A|}\mathbb{P}\{X^{(\ell+1)}\subset A\}^{\alpha} \nonumber\\
&=\mathbb{P}\{X^{(\ell+1)}\subset B\}^{\alpha}-\left(\sum_{i\in B} s_i\right)^{\alpha}+\sum_{A\subset B}(-1)^{|B|-|A|}\left(\sum_{i\in A}s_i \right)^{\alpha} \nonumber\\
&=\left(\sum_{i\in B} s_i + c_2 \epsilon\right)^{\alpha}-\left(\sum_{i\in B} s_i\right)^{\alpha}+\sum_{A\subset B}(-1)^{|B|-|A|}\left(\sum_{i\in A}s_i \right)^{\alpha}.
\end{align}

For $\epsilon=0$, it follows from Proposition~\ref{prop:formultipleintervals} that $r_{\ell+1,B}(\alpha)=0$ for $\alpha=n-\ell-1,n-\ell$ and $r_{\ell+1,B}(\alpha)<0$ for $n-\ell-1<\alpha<n-\ell$. Therefore for small enough $\epsilon>0$ one can have $r_{\ell+1,B}(\alpha)<0$ for some $n-\ell-1<\alpha<n-\ell$. Thus for small enough $\epsilon >0$, $\XX_{\alpha}^{(\ell+1)}$ does not exist for some $\alpha_j\in (j,j+1) $ and $\forall n-(\ell+1)\leq j\leq n-2$.

Fix such a small $\epsilon>0$. We show that if $|B|\geq n-(\ell+1)+2$, then $r_{\ell+1,B}(\alpha)>0$ for $\alpha\in(j-\delta,j+\delta)$ for some small $\delta>0$ and  $\forall j\in \mathbb{N}$. If $|B|\geq n-\ell+2$, this is true for small enough $\epsilon>0$ as $\XX^{(\ell+1)}$ is defined by perturbing $\XX^{(\ell)}$ and by the induction hypothesis $r_{\ell,B}(\alpha)>0$ for $\alpha\in(j-\delta,j+\delta)$ for some small $\delta>0$ and  $\forall j\leq n-2$. Suppose $|B|=n-(\ell+1)+2$. In this case, $r_{\ell+1,B}(\alpha)$ is given by \eqref{eq:multipleintervalproof1}. Note that the function $\alp \mapsto \sum_{A\subset B}(-1)^{|B|-|A|}(\sum_{i\in A }s_i)^{\alpha}$ is a continuous function which is non-negative at positive integers $m\le n-2$ (by Proposition~\ref{prop:formultipleintervals}) and there exists $\eta>0$ such that 
\[\left(\sum_{i\in B} s_i + c_2 \epsilon\right)^{\alpha}-\left(\sum_{i\in B} s_i\right)^{\alpha}>\eta \]
for all $\alpha\in[1/2,n-1]$. Hence for each positive integers $m\le n-2$ we have $r_{\ell+1,B}(\alpha)>0$ in an interval around $m$.

Finally, we show that $\XX_{\alpha}^{(\ell +1)}$ exists when $\alpha\in [j,j+\delta)$, $\forall n-(\ell+1)\leq j\leq n-2$. First consider $n-\ell\leq j\leq n-2$. Since $\XX_{\alpha}^{(\ell)}$ exists when $\alpha\in [j,j+\delta)$ and $\XX^{(\ell+1)}$ is defined by perturbing $\XX^{(\ell)}$, we have that $\XX_{\alpha}^{(\ell +1)}$ exists for $\alpha\in [j,j+\delta_0)$ for some $\delta_0>0$. Next consider $j=n-(\ell+1)$. We already proved that if $|B|\geq n-\ell+1$, then $r_{\ell+1,B}(\alpha)>0$ for $\alpha\in[j,j+\delta)$ for some small $\delta>0$. On the other hand, if $|B|< n-\ell+1$, then using Proposition~\ref{prop:formultipleintervals} we have
\[
r_{\ell+1,B}(\alpha)=\sum_{A\subset B}(-1)^{|B|-|A|}\left(\sum_{i\in A}s_i \right)^{\alpha} \ge 0
\]
for $\alp\ge n-\ell-1$.

The proof is complete by induction.
\end{proof}

\section{Proof of Theorem~\ref{thm: TJ trick}}\label{sec: proof of TJ trick}
%\begin{proof}[Proof of Theorem ~\ref{thm: TJ trick}] 
If $f: \N \rightarrow(0,\infty)$ is a c.m.\ sequence, then by Hausdorff's moment sequence theorem (see  Proposition $6.11$ of Chapter~$4$ of~\cite{BCRR2012harmonic}), there is a corresponding probability measure $\mu_f$ on $[0,1]$ whose moment sequence $m_f:\N \rightarrow \R$ defined by $m_{f}(k)=\int\limits_{0}^{1}x^k d\mu_{f}$, is up to scaling, equal to $f$. Note that if $X_1,X_2$ are i.i.d.\ random variables with $m_{f}$ as their moment sequence, then the random variable $Y=X_1X_2$ has moment sequence ${m_{f}^2}$. Thus $f^2$ is a c.m.\ sequence by Hausdorff's moment sequence theorem.  Proceeding similarly, one can prove that if $f$ is a c.m.\ sequence, then $f^k$ is a c.m.\ sequence, $\forall k\in \N$.

We now construct a c.m.\ sequence $f$ such that $f^{\alp}$ is not c.m.\ for $\alp\notin\N$. For any $x\in[0,1]$, let $\delta_x$ denote the Dirac measure at $x$. Fix $x\in (0,1)$ and define the probability measure $\mu= \frac{1}{2}\delta_1+\frac{1}{2}\delta_x$. Fix any $\alp\notin\N$. Let $f$ be defined as $f(k)=\int\limits_{0}^{1}y^k d\mu(y)$. Note that $ f$ is a c.m.\ sequence, by Hausdorff's moment sequence theorem. As $f$ is a sequence in $[0,1]$, $f^{\alpha}$ is also a sequence in $[0,1]$. If $f^{\alp}$ is c.m.\ then it has to be a moment sequence of some probability measure on $[0,1]$. We now use the well-known fact that, if $h$ is a moment sequence of a probability measure on $[0,\infty)$ then the infinite array $\{m_{ij}\}_{i,j\geq 0}$ given by $m_{ij
}=h(i+j-2)$ is positive semi-definite (Lemma $1.19$ of~\cite{AK}).  If we show that the array given by $m_{ij}=f^{\alpha}(i+j-2)$ is not p.s.d.\ then we have that $f^{\alp}$ is not c.m.\ sequence. Consider the matrix $[m_{ij}]_{1\leq i,j\leq n}$. By using Theorem $1.1$ of~\cite{jain} with $x_i=x^i$, we have that the matrix $[m_{ij}]_{1\leq i,j\leq n}$ is not p.s.d.\ This shows that $f^{\alp}$ is not p.s.d.\ and the proof of the first part is complete.

If $ g:(0,\infty)\rightarrow [0,\infty)$ is a c.m.\ function then by Bernstein's theorem (see Theorem $6.13$ of Chapter~$4$ of~\cite{BCRR2012harmonic}), there is a corresponding probability measure $\mu_g$ on $[0,\infty)$ whose Laplace transform defined by $\mathcal{L}_g(t)=\int\limits_{0}^{\infty}\exp(-tx) d\mu_{g}(x)$, is up to scaling, equal to $g$.  Note that if $X_1,X_2$ are i.i.d.\ random variables with $\mathcal{L}_g(t)$ as their Laplace transform, then the random variable $Y=X_1+X_2$ has Laplace transform $\mathcal{L}_g^2(t)$. Thus $g^2$ is a c.m.\ function. Proceeding similarly, we can show that if $g$ is a c.m.\ function, then $g^k$ is a c.m.\ function, $\forall k\in \N$.

We now construct a c.m.\ function $g$ such that $g^{\alp}$ is not c.m.\ for $\alp\notin\N$. Choose a probability measure on $[0,\infty)$, $\mu=\frac{1}{2}\delta_0+\frac{1}{2}\delta_y$ where $y\neq 0$. Fix any $\alp\notin\N$. Let $Y$ be a random variable with $Y\sim\mu$, and $g(t)=\E[\exp(-tY)]$. If $g^\alpha$ is a c.m.\ function, then by Bernstein theorem, $(\E[\exp(-tY)])^{\alp}=\E[\exp(-tZ)]$ for some random variable $Z$ and for all $t>0$. But $\{(\E[\exp(-kY)])^{\alp}\}_{k\geq 1}$ is the $\alp$-th power of moment sequence of the random variable $\exp(-Y)\in [0,1]$. As seen in the first part, there cannot be a random variable $\exp(-Z)$ whose moment sequence is the sequence $\{(\E[\exp(-kY)])^{\alp}\}_{k\geq 1}$. This gives a contradiction. Thus $g^{\alpha}$ is not c.m.\ function for any $\alpha\notin\N$.

\subsection*{Acknowledgement.}
The authors thank Manjunath Krishnapur for suggesting the questions addressed in this article and for several helpful discussions without which this article could not have been possible. They thank Arvind Ayyer for a helpful discussion. The second author acknowledges the support of Indian Institute of Science (C. V. Raman postdoctoral fellowship), National Board for Higher Mathematics, India (NBHM postdoctoral fellowship), and Department of Science and Technology, India (INSPIRE Faculty Fellowship, IFA22-MA176).
\bibliographystyle{abbrvnat}
\bibliography{biblio_rand_set}

%m,n counterexample

%$n=5,m=2$.

%Let $X$ be uniformly chosen random $m-$ element subset of $[n]$.

%$p_{ij}=1/10$.

%\begin{align*}
%    q_{1}=q_{2}=\dots=q_{5}=0\\
%    q_{12}=(1/10)^{\alpha}\\
 %   q_{123}=(3/10)^{\alpha}-3(1/10)^{\alpha}\\
  %  q_{1234}=(6/10)^{\alpha}-4(3/10)^{\alpha}+6(1/10)^{\alpha}\\
   % q_{12345}=1-5.(6/10)^{\alpha}+10.(3/10)^{\alpha}-10.(1/10)^{\alpha}
%\end{align*}
%
%Powers of void probabilities of $X$ are void %probabilities of another random subset if and %only if above equations are all positive.
%
%Plotting the remaining three functions, we find %that valid values of $'\alpha'$ form two %disjoint intervals.
%

%This contradicts the 'if and only if' statement %that $X_{\alpha}$ exists if and only if %$\alpha\geq\alpha^*$.

\end{document}

%% file: mypreamble_2018.tex
\usepackage{amsfonts}
\usepackage{amssymb}
\usepackage{amsmath}
\usepackage{amsthm}
\usepackage{amsbsy}
\usepackage{amssymb} % Without this \qed does not get filled square
\usepackage{verbatim}
\usepackage{bm} %for some types of bold letters
\usepackage{paralist} %Need to use \inparaenum 
 \usepackage{color}
 \usepackage{mathrsfs}
 \usepackage{graphicx}
\usepackage{hyperref}
\newcommand{\E}{\mathbf{E}}
\def\P{\mathbf{P}}
\def\to{\rightarrow}

\def\mb{\mbox}

\def\l{\left}
\def\r{\right}
\def\<{\langle}
\def\>{\rangle}

\newcommand{\ba}{\[\begin{aligned}}
\newcommand{\ea}{\end{aligned}\]}
\newcommand\mnote[1]{} %off
\newcommand{\beq}[1]{\begin{equation}\label{#1}}
\newcommand\eeq{\end{equation}}
\newcommand\ben{\begin{equation}}
\newcommand\een{\end{equation}}
\newcommand\bes{\begin{eqnarray*}}
\newcommand\ees{\end{eqnarray*}}
\newcommand\besn{\begin{eqnarray}}
\newcommand\eesn{\end{eqnarray}}

\def\bthm{\begin{theorem}}
\def\ethm{\end{theorem}}
\def\bdefn{\begin{definition}}
\def\edefn{\end{definition}}
\newcommand{\benu}{\begin{enumerate}\setlength\itemsep{6pt}}
\newcommand{\beit}{\begin{itemize}\setlength\itemsep{3pt}}
\def\eenu{\end{enumerate}}
\def\eeit{\end{itemize}}
\def\beds{\begin{description}}
\def\eeds{\end{description}}
\def\bepr{\begin{problem}}
\def\eepr{\end{problem}}

\def\bprf{\begin{proof}}
\def\eprf{\end{proof}}
\def\berk{\begin{remark}}
\def\eerk{\end{remark}}
\def\bex{\begin{exercise}}
\def\eex{\end{exercise}}
\def\beg{\begin{example}}
\def\eeg{\end{example}}

\def\BB{{\mathcal B}}

\def\DD{{\mathcal D}}
\def\FF{{\mathcal F}}

\def \KK{{\mathcal K}}
\def\LL{{\mathbb L}}

\def\XX{{\mathcal X}}
\def\YY{{\mathcal Y}}

\def\N{\mathbb{N}}
\def\R{\mathbb{R}}

\def\S{\mathbb{S}} % for sphere

\newcommand{\sm}{{\raise0.3ex\hbox{$\scriptstyle \setminus$}}}

\def\alp{\alpha}

\renewcommand\phi{\varphi}

\theoremstyle{plain} %documentation says there are only three styles
    \newtheorem{theorem}{Theorem}
    \newtheorem{lemma}[theorem]{Lemma}
    \newtheorem{proposition}[theorem]{Proposition}
    
    \newtheorem{claim}[theorem]{Claim}

\theoremstyle{definition} % For roman text in the body
    \newtheorem{definition}[theorem]{Definition}

    \newtheorem{exercise}[theorem]{Exercise}
    \newtheorem{problem}[theorem]{Problem}
        \newtheorem{remark}[theorem]{Remark}
    \newtheorem{example}[theorem]{Example}